\documentclass[11pt]{article}
\usepackage{latexsym}
\usepackage{amsmath,amsthm,amsfonts,amssymb} %, mathrsfs} %, wasysym, calligra}
\usepackage{mathtools}% http://ctan.org/pkg/mathtools
\usepackage{hyperref}

\newcommand{\ang}{\sphericalangle}

\newcommand{\Act}{{\rm Act}}

\newcommand{\MCG}{{\rm MCG} (\Sigma)}

\newtheorem{thm}{Theorem}
\newtheorem{theo}[thm]{Theorem}

\newtheorem{defi}[thm]{Definition}

\newtheorem{prop}[thm]{Proposition}

\newtheorem{lemma}[thm]{Lemma}

\makeatletter
\@tfor\next:=abcdefghijklmnopqrstuvwxyzABCDEFGHIJKLMNOPQRSTUVWXYZ\do{%
  \def\command@factory#1{%
    \expandafter\def\csname cal#1\endcsname{\mathcal{#1}}
    \expandafter\def\csname bb#1\endcsname{\mathbb{#1}}
    \expandafter\def\csname rm#1\endcsname{\mathrm{#1}}
  }
 \expandafter\command@factory\next
}
\makeatother

\newcommand{\G}{\Gamma}
\newcommand{\g}{\gamma}

\newcommand {\onto} {\twoheadrightarrow}

\title{The normal closure of  big Dehn twists, and  plate
  spinning with rotating
  families }
\author{Fran\c{c}ois Dahmani}
\date{}
\begin{document}
\maketitle
\begin{abstract} We study the normal closure of a big power of one or several Dehn
  twists in a Mapping Class Group.  We prove that it has a 
  presentation whose relators consist only of commutators between twists of
  disjoint support, thus answering a question of Ivanov.
  Our method is to use  the theory of projection
  complexes of Bestvina Bromberg and Fujiwara, together with the theory
  of  rotating families, simultaneously on several spaces.  
\end{abstract}
\setcounter{tocdepth}{2}
\tableofcontents

\section*{Introduction}

Consider a closed orientable surface $\Sigma$ of negative Euler
characteristic. 
 The Mapping Class Group of $\Sigma$, denoted by $\MCG$, is the quotient of the group
of orientation-preserving homeomorphisms by the path-connected component
of the identity. A classical theorem of Dehn and Nielsen indicates a
natural isomorphism between this
group and a subgroup of index $2$ of the outer automorphism group of $\pi_1(\Sigma)$.

As Riemann uniformisation theorem makes  
$\pi_1(\Sigma)$ act  as a lattice on the hyperbolic plane, one can argue that $\MCG$ is (in a sense) 
some hyperbolic
analogue of $SL_2(\bbZ)$ which is of index $2$ in
the automorphism group of $\bbZ^2$, a lattice in the euclidean plane. 

However, contrarily to $SL_2(\bbZ)$, some nontrivial elements of $\MCG$ have large centraliser. For instance, consider
a simple closed curve $\alpha$ on $\Sigma$, a tubular neighborhood of
it $\alpha^{(t)} \simeq  [-\epsilon, \epsilon] \times \alpha
\hookrightarrow \Sigma$ and define a (simple) Dehn twist
$\tau$ as the identity in $\Sigma \setminus \alpha^{(t)}$, and as a
full twist on $\alpha^{(t)}$, namely, identifying $\alpha$ with $S^1$,
the map $[(\eta, e^{i\theta}) \mapsto (\eta, e^{i(\theta +
  \frac{(\eta+\epsilon)\pi}{\epsilon } )} )]$.  A Dehn twist will
obviously commute with any mapping class whose support is disjoint
from this tube, and therefore with a lot  of other Dehn twists.  By
a theorem of Dehn,  $\MCG$   is generated by Dehn twists around simple closed
curves, thus by an intricate set of generators linked by
commutation relations, but also braid relations and lantern
relations. 
These differences can lead to modify the expected analogy with
the euclidean case in order to include  $SL_n(\bbZ)$ for $n\geq
3$ (generated by elementary matrices). 

 Thurston, and Nielsen, (see the discussion and references in \cite{HT}) classified  mapping classes into three
cases, those of finite order, those that are reducible in the sense
that they have infinite order and that some nontrivial power  preserves the homotopy class of a simple
closed curve, and finally the pseudo-Anosov. The pseudo-Anosov mapping
classes happen to be
the hyperbolic isometries of an action of $\MCG$ on an important
graph, the curve graph of $\Sigma$, which is Gromov hyperbolic \cite{MM}.  They
are, 
in many ways, the witnesses that some phenomena of rank one
happen in $\MCG$  that
are similar to the structure of $SL_2(\bbZ)$, and its action  on the modular tree. On the
other hand, Dehn twists are as reducible as it is possible to be. They
are, or should be,  the witnesses of some phenomena of higher rank,
similar to the structure of $SL_n(\bbZ)$ for $n\geq
3$.

Here is an illustration of the difference of behaviors. If one considers a finite collection of pseudo-Anosov
elements, one can  show that, after taking suitable powers,  the group
they generate is free \cite{Ifree, McC}. This is a  ping-pong argument, for instance on
the boundary of Teichm\"uller space, or on the curve graph. If one considers a finite collection of Dehn
twists around simple closed curves, then Koberda \cite{K} proved the
beautiful ping-pong result that the group generated by some powers of these Dehn
twists is a right angled Artin group: a group whose presentation over
the given generating set is a collection of commutators, the obvious
ones (two Dehn twists commute if their curves are disjoint).

The case of normal subgroups is our interest.
If $n\geq 3$, by Margulis' normal subgroup
theorem, all normal subgroups of $SL_n(\bbZ)$  are finite or of finite
index. In $SL_2(\bbZ)$ it is not the case: this group is virtually
free, and has uncountably many non-isomorphic quotients.

  It is a
natural question to ask whether (and how) these phenomena are seen in
$\MCG$. What can be the normal closure of a power of a pseudo-Anosov,  the normal
closure of a power of a  Dehn twist, and the group generated by all $k$-th powers
of all simple Dehn twists~?  Farb and Ivanov 
asked this question in the case of a pseudo-Anosov (respectively
\cite[\S 2.4]{F} and \cite[\S 3]{I}),  attributing it
 to Long, McCarthy, and Penner.  Ivanov also asked what he calls the deep
relation question \cite[\S 12]{I}, that is  
whether  all relations among certain powers  of Dehn twists must derive from
obvious commutation relations.

In \cite[\S 5]{DGO}, we answered the first question: there is an integer $N=N(\Sigma)$
such that for any pseudo-Anosov mapping class $\gamma$, the normal
closure $\langle\langle \gamma^N\rangle\rangle_{\MCG}$ is free, and
consists only of pseudo-Anosov elements and the identity. This
is in line with what happens in $SL_2(\bbZ)$, for each infinite order element.

We are interested in the question of the closure of a power of a Dehn twist, and
in the group generated by certain powers of all (simple) Dehn twists, as
in Ivanov's deep relation problem. A
naive expectation along the lines of the analogy with  $SL_n(\bbZ)$,
and the Margulis normal subgroup theorem, could be to expect such 
normal subgroups to be a finite
index subgroup. Whereas it is the case for squares of Dehn twists
\cite{H},  it is not the case for large powers (see \cite{H}, \cite{Fu},
\cite[6.17]{Cou}, see also \cite{S} and \cite{Mas} for the case of
powers of half-Dehn twists on punctured spheres). 
Another
expectation could be,   in light of the finite-type situation, and
ping pong arguments, to expect  infinitely generated right angled
Artin groups. Again, this is not the case in general (see \cite{CLM}
and \cite{BM}; Brendle and Margalit proved restrictions on the
automorphism group of certain of these normal subgroups, that
forbid them to be  right angled Artin groups). However, we indeed
prove that there is no need of  relations other than the obvious ones.

\begin{theo}\label{theo;intro}

For every orientable closed surface $\Sigma$, there is an integer $N_0$ such that for any $N$ multiple of $N_0$:

\begin{itemize}\item for any Dehn twist $\tau$, the normal
closure of  $\tau^N$ in the Mapping
Class Group of $\Sigma$  has a partially
commutative presentation, built on an infinite set of generators that are
conjugates of $\tau^N$, so that the relators are commutations
between pairs of conjugates of  $\tau^N$ that have disjoint underlying
curves. 

\item the group generated by all $N$-th powers of all simple Dehn
  twists has a partially
commutative presentation,  built on an infinite set of generators that
are $N$-th powers of Dehn twists, and whose relators are commutations
between pairs of conjugates of the generators that have disjoint underlying
curves. 
\end{itemize}

\end{theo}

The difference with an  infinitely generated right angled
Artin group is that some elements in the commutator relators are not
in the generating set, but merely conjugates of elements in the
generating set.  We recover that  the normal closure is far
from being of finite index in $\MCG$, for instance because it has
 abelianisation of infinite rank  (the
relators being in the derived subgroup of the free
group over the set of generators). 

  In our point of view, this result above, and its departure
from the complexity of normal subgroups of $SL_n(\mathbb{Z})$ for
$n\geq 3$ (granted by Margulis normal subgroup theorem) reinforce
\cite{Fu,Cou} in 
witnessing a dent in the analogy between $\MCG$ and  $SL_n(\mathbb{Z})$. It 
also   answers  Ivanov's question on deep relations.

Let us discuss the proof of this theorem.

In \cite{DGO} the structure of the normal closure of a big
pseudo-Anosov was studied with the help of rotating families. Consider
$G$ a group acting by isometries on a space $X$. A
rotating family in $G$ on $X$  is a collection of subgroups (the
rotation groups), that is
closed under conjugacy, such that each of them fixes a certain point in
$X$ (thus inducing some kind of rotation around this point). Take
$\rho$ in one of these subgroups,  
fixing $c$. One may
measure an analogue of the angle of rotation of $\rho$ 
by taking $x$ at distance $1$ from $c$, and  measuring the infimal length
between $x$ and $\rho x$  of paths outside the ball of radius $1$
around $c$.     If $X$ is Gromov-hyperbolic (for a small hyperbolicity
constant), if the  fixed point of the different rotation groups are
sufficiently far from each other, and if the angles of rotations are
sufficiently big, the group generated by all the rotation groups is a
free product of a selection of them.  In \cite{DGO} we applied this theory to the
action of $\MCG$ on a cone-off of the curve graph of $\Sigma$. The
rotation groups were the conjugates of the big pseudo-Anosov considered.

The rotating family argument can be explained as follows. One analyses
the structure of groups generated by more and more rotation groups, to
discover that they arrange as a sequence of free products. Starting
from a quasi-convex set $W$ (that will change over time) that is at
first a small ball around a single fixed point of
a single rotation group, one sets $G_W$ the group generated by the rotation groups whose centers are in $W$, and one makes $W$ grow until it (almost)
touches another center of rotation, for some other group.
Call $S$ a $G_W$-transversal of the newly approached centers of rotation. 
 Then one
\emph{unfolds} $W$ into $W'$ by taking its images by the group $G_{W'}$ (thus generated by the new
rotations, and the rotation already with center in $W$). Because of
hyperbolicity, and of largeness of angles of rotations involved, the resulting
space is still quasi-convex, with  almost the same constant -- with
a little repair, it has the same quasi-convexity constant indeed.
Actually  $W'$ has the
structure of a tree whose vertices are the images of $W$ by the group $G_{W'}$,
and the images of points in $S$ by $G_{W'}$, thus giving by Bass-Serre duality the structure of
free product of $G_W$ and the rotation groups around points of $S$  (edge stabilizers are trivial since no element can fix two
different centers of rotation).  Then, one takes the new $W$ as $W'$ and start over. In the direct
limit, the group generated by all rotations has been described as a
free product of a selection of rotation  groups.

In \cite{BBF}, Bestvina Bromberg and Fujiwara, using a system of
subsurface projections,  discovered that there is
a normal finite index subgroup $G_0$ of $\MCG$ that acts on some spaces
quasi-isometric to trees, and on which Dehn twists behave like large
rotation subgroups. It has been  observed by several people that this implies
that the normal closure of a certain power of a Dehn twist in $G_0$ is
free, using the argument of \cite{DGO}. However, it is far from obvious how
to promote this structural feature to the normal closure in $\MCG$. 

In this paper, we use several quasi-trees as above, one for each left
coset of $G_0$ in $\MCG$.  The group $G_0$ acts on each of them, but
its action is twisted by the automorphism of $G_0$ that is the
conjugation by elements $g_i, i=1, \dots, m$ realising a transversal of $G_0$
in $\MCG$.  If $\tau^N$ is a Dehn twist in $G_0$,  the
normal closure of $\tau^N$ in $\MCG$ equals the normal closure of the
collection $\{ g_i \tau^Ng_i^{-1},\,  i=1,\dots m  \}$ in $G_0$.  Each $g_i
\tau^Ng_i^{-1}$ is a legitimate rotation on the quasi-tree associated
to $g_i$. 

The argument of \cite{DGO} is then performed simultanously on each of the
$m$ quasi-trees. Instead of one convex subset that grows, and gets
unfolded in a hyperbolic space,  we have   $m$ convex sets $\calW_1,
\dots, \calW_m$ in the $m$
quasi-trees. Each of them is invariant by the group generated by the
rotations around rotation points in all of them.  One looks for a
rotation point $R$ that is nearby one of these sets, and in a certain
sense, nearby all of them (although they do not live in the same
quasi-trees, this still makes sense in the framework of projection
systems). Then, one unfolds our convex sets in all coordinates $i=1,
\dots, m$.
A funny phenomenon happens. The unfolding in the coordinate of $R$
provides a nice tree, as the argument of \cite{DGO}, and the convexity
 of the result is quantitively very good.  This tree gives
 the structure of the new group by Bass-Serre duality, and reveals
 that only
 commutation relations are involved. There is no reason that the unfolding in all other
 coordinates produces something resembling  a tree, and could in
 principle destroy the convexity of $\calW_j$. However, using
 the properties of the projection system, we show that the result is
 still somehow convex (less convex than before though). 
   The game is then to unfold in the different
quasitrees at regular intervals of time in the process, and to control the degradation
of the convexity so that the repair can wait until a new unfolding
occurs. It is a game of plate spinning.

The quasi-trees that we will use come from projection complexes
defined in \cite{BBF}. We wrote the argument in
this axiomatic language, to avoid dealing with
useless hyperbolicity constants. In the end, even if the spaces are
indeed quasi-trees, this fact does not appear in the argument. The
axioms of projection systems are extensively used though, and they
contain the information that the geometric space is a quasi-tree. We
will thus prove a similar statement as Theorem \ref{theo;intro},
namely Theorem \ref{theo;main}, that gives the structure of groups
generated by composite rotating families. There is actually more
information coming from this composite rotating family structure, as
for instance the Greendlinger property (see Definition \ref{def;CW}), that describes
how an element in the group can be shortened in some coordinate of the
composite projection system.

\numberwithin{thm}{section}

\section{Composite projection systems}
     \subsection{Projection systems}
        Let us recall a part of the axiomatic construction of
        \cite{BBF}.

        \begin{defi} (\cite{BBF}) 

          A projection
        system is a set $\bbY$, with a constant $\theta>0$, 
        and for each $Y\in       \bbY$,  a 
        function $ \left(d_Y^\pi: \bbY\setminus\{Y\} \times
        \bbY\setminus\{Y\} \to \bbR_+\right)$ satisfying the following axioms: 
        \begin{itemize}
          \item       symmetry ($d_Y^\pi (X,Z) = d_Y^\pi (Z,X)$ for all $X,Y,Z$), 
          \item triangle inequality
       ($d^\pi_Y(X,Z) + d_Y^\pi(Z,W) \geq d_Y^\pi(X,W) $ for all $X,Y,Z, W$), 
         \item Behrstock
       inequality ($ \min\{ d^\pi_Y(X,Z), d_Z^\pi (X,Y) \} \leq \theta
       $ for all $X,Y,Z$), 
          \item properness ($\{Y, d_Y^\pi (X,Z) >\theta\}$ is
            finite for all $X,Z$). 
          \item In
       this work one also assume the separation ($d_Y^\pi (Z,Z) \leq
       \theta$ for all $Z,Y$).
       \end{itemize}
       \end{defi} 

       Observe that if the axioms are true for some
       $\theta$ they hold for all larger $\theta$.

       From this rudimentary axiomatic set, Bestvina Bromberg and
       Fujiwara manage to extract meaningful geometry, by modifying
       the functions $d_Y^\pi$ into some functions $d_Y$, that satisfy many more
       properties, usually encapsulated in the statement that the
       projection complex of $\bbY$, for a suitable  parameter
       $K$
       is a quasi-tree. 

       One should  think of $d_Y$ (or
       $d^\pi_Y$)  as an
       angular measure between $X$ and $Z$ seen from $Y$.  The axioms
       fit in this viewpoint:  the Behrstock inequality  says that if
       the angle at $Y$ between $X$ and $Z$ is large, then from the point of view of $Z$, the items $Y$ and $X$
       look aligned.

       Let us review very quickly the procedure of \cite{BBF} to
       produce the functions $d_Y$. 
       Given $\theta$ for which the axioms hold,  \cite{BBF} define $\calH(X,Z)$ to be
       the set of pairs $(X',Z')$ such that both $d_X^\pi$ and $d_Z^\pi$
       between them is strictly larger than $2\theta$, and one also include the
       pairs $(X,Z')$ if $d_Z^\pi (X,Z' ) >2\theta$, symmetrically  the
       pairs $(X',Z)$ if $d_X^\pi (X',Z ) >2\theta$, and finally the
       pair $(X,Z)$ itself.

       Then $d_Y(X,Z) $ is defined to be the infimum of $d_Y^\pi$ over
       $\calH(X,Z)$.

       For all $K$, $\bbY_K(X,Z)$ denotes the set $\{Y, d_Y(X,Z) \geq K\}$.

       \cite[Theorem 3.3]{BBF} states that there exists $\Theta$ and
       $\kappa\geq \theta$
       depending only on  $\theta$, such that for all $X,Y,Z, W$:

       \begin{itemize}
         \item (Symmetry)  $d_Y(X,Z) = d_Y(Z,X)$
         \item (Coarse equality) $d_Y^\pi -\kappa \leq d_Y \leq d_Y^\pi$
         \item (Coarse triangle inequality) $d_Y(X,Z) + d_Y(Z,W) \geq d_Y(X,W) -\kappa$
         \item (Behrstock inequality) $\min\{ d_Y(X,Z) , d_X(Y,Z) \}\leq \kappa$
         \item (Properness) $\{ V, d_V(X,Z) >\Theta\}$ is finite
         \item (Monotonicity) If $d_Y(X,Z) \geq \Theta$ then both
           $d_W(X,Y), d_W(Z,Y)$ are at most $d_W(X,Z)$.
         \item (Order) $\bbY_\Theta (X,Z) \cup \{X, Z\}$ is totally
           ordered by an order $ \dot{<}$ such that $X$ is lowest, $Z$
           is greatest, and if $Y_0 \dot{<} Y_1 \dot{<}Y_2$, then 
           $$  d_{Y_1}(X,Z) -\kappa  \leq d_{Y_1} (Y_0, Y_2)  \leq  d_{Y_1}(X,Z), $$
           and   
           $$ d_{Y_0} (Y_1, Y_2) \leq \kappa, \quad  d_{Y_2}(Y_1, Y_0)   \leq \kappa.  $$ 
        
       \end{itemize}

       Then choosing $K$ larger than $\Theta$, the projection complex
       $\calP_K(\bbY)$ is defined as follows: it is a graph whose
       vertices are the elements of $\bbY$ and where $X,Z$  span an
       edge if and only if $\bbY_K(X,Z) =\emptyset$. Then
       \cite[Thm. 3.16]{BBF} states that for sufficiently large $K$,
       $\calP_K(\bbY)$ is connected and quasi-isometric to a tree for
       its path metric.

     \subsection{Composite projection systems}
         In this work, we are concerned with a composite situation.

         \subsubsection{Definitions, and projection complexes}

         Let $\bbY_*$ be the disjoint union of finitely many countable
         sets
         $\bbY_1, \dots, \bbY_m$. Their indices $i=1, \dots, m$  are called the coordinates.         Given $Y\in \bbY_*$, denote by $i(Y)$ its coordinate: 
         $Y\in \bbY_{i(Y)}$.

         \begin{defi}\label{def;CPS}
         A composite projection system on a countable set $\bbY_* =
         \sqcup_{i=1}^m \bbY_i$ 
         is the data of a
         constant $\theta>0$,  of a family of subsets  $\Act(Y)\subset
         \bbY_*,  Y\in
         \bbY_*$  (the active set for $Y$) such that $\bbY_{i(Y)} \subset\Act(Y)$,     
          and of a family of functions
         $d^\pi_Y : ( \Act(Y)\setminus \{Y\} \times \Act(Y) \setminus \{Y\})  \to \bbR_+$, satisfying the
         symmetry, the triangle inequality, the Behrstock inequality for
         $\theta$ whenever both quantities are defined, 
         the properness for $\theta$ when restricted to each
         $\bbY_i$,  
         the separation for
         $\theta$, and also   three other properties related to the
         map $\Act$:  
         \begin{itemize}
           \item (symmetry in action) $X\in \Act(Y)$ if and
         only if $Y\in \Act(X)$, 
           \item (closeness in inaction) if   $X\notin \Act(Z)$, for
         all $Y \in \Act(X) \cap \Act(Z)$, $ d_Y^\pi( X,Z )\leq \theta
         $ 
         \item (finite filling)  for all $\calZ\subset \bbY_*$, there
           is a finite collection 
            of elements $X_j$ in $\calZ$ such that $\cup_j
            \Act(X_j)$ covers  $\cup_{ X\in \calZ} \Act(X)$.
         \end{itemize}
         \end{defi}

         The closeness in inaction  can be understood as a
         complement to Behrstock inequality: ``if $ d_Y^\pi( X,Z )> \theta
         $, then  $ d_X^\pi( Y,Z )$ is {\it defined} and is less than $\theta$''.

         Applying \cite{BBF} (as recalled in the previous subsection) 
         we  get, for each coordinate
         $i\leq m$, and  for a suitable choice of $\theta$,  a
         modified function $d_Y : \bbY_{i} \times \bbY_i       \to
         \bbR_+$.  This function is unfortunately not defined on
         $\Act(Y)\setminus \bbY_i$, but $d^\pi_Y$ is defined on it,
         and thus we choose to define $d^\ang_Y (X,Z)$ to be $d_Y$ if
         both $X, Z$ are in $\bbY_i$ and $d^\pi_Y$ otherwise.

         We then define $\bbY^j_M (X,Z) = \{Y\in \bbY_j \cap \Act(X)
         \cap \Act(Z),  d^\ang_Y( X,Z) \geq M\}$.  The elements
         $X,Y,Z$ need not be in the same coordinate.

         In the following we first choose $\theta$ such that the
         construction of \cite{BBF} applies for all coordinates
         $\bbY_i$, and this provides the constants $\Theta$  and
         $\kappa$ (suitable for all coordinates).

         Then we choose $c_* >1000 (\Theta +\kappa)$, and $\Theta_P =    c_*
         +21m\kappa $. One can choose  $K > \Theta_P$ 
         sufficiently large to get  quasi-trees
         in all coordinates, but this is not important for us. 

         Finally,  choose 
         $\Theta_{Rot} >  2c_* +  2\Theta_P + 20 (\kappa +\Theta)   $ 
         for later purpose.
        
         To keep track of the constants, it is worth keeping in mind
         that 
         $$\Theta_{Rot} >> 2 \Theta_P >> 2c_* >> 20 (\Theta +\kappa) >>\theta.$$

         \subsubsection{Group in the picture}

         An \emph{automorphism} of a composite projection system is a map
         $\psi: \bbY_*\to \bbY_*$ 
         \begin{itemize}
            \item that induces a bijection on each
         $\bbY_i$, 
            \item that sends $\Act(Y)$ to $\Act(\psi(Y))$, 
            \item such  that for all $Y$, and all $X,Z\in \Act(Y)$, $d_Y^\ang (X,Z) =
              d_{\psi(Y)}^\ang(\psi(X), \psi(Z))$.            
          \end{itemize}

          A \emph{rotation} around $X\in \bbY_*$  in a composite projection
          system  $\bbY_*$ is an automorphism $\psi$ such that
          $\psi(X) = X$, and such that for all $Y\in \bbY_* \setminus
          \Act(X)$, and for all $ W, Z \in \Act(Y)$,   $\psi(Y)=Y$,
          and  $d_Y^\ang (W,Z) =d_Y^\ang (\psi(W),\psi(Z))$.

         Let us now assume that a group $G$ acts on the composite
         projection system by automorphisms.

         Let us denote by $G_X$ the stabilizer of $X\in \bbY_*$.

         We say that a subgroup $\Gamma_X<G_X$   has \emph{proper
           isotropy}    
         if for all $N>0$ there is 
         a finite subset $F(N)$ of $\Gamma_X$  such that
         if $\gamma \in  \Gamma_X\setminus F(N)$, and if $Y\in \Act(X)$, 
         then $d_X^\pi(Y, \gamma Y)
         >N$.

         \subsubsection{Betweenness and orbit estimates}

         \begin{lemma} (Betweenness is transitive)\label{lem;bet_trans}
           
           If $d^\ang_Y (X,Z) >2\kappa$ and  $d^\ang_Z (Y,T)
           >2\kappa$, then $Z$ is in $\Act(X)$ and $d_Z^\ang(X,T )
           \geq  d^\ang_Z (Y,T)  - 2 \kappa$.

 If $d^\ang_Y (X,Z) >10\kappa$ and  $d^\ang_Z (X,T)
           >10\kappa$, then $d^\ang_Y (X,T) \geq  d^\ang_Y (X,Z) -2\kappa$.
         \end{lemma}

         \begin{proof}
           By Behrstock inequality, 
           one has $d^\ang_Z (X,Y)  \leq \kappa$ in both cases. For
           the first implication,   by triangular
           inequality, $d_Z^\ang (X,T ) \geq d_Z^\ang(Y,T ) -
           d_Z^\ang(X,Y ) - \kappa $.

           For the second implication, 
           $d^\ang_Z( Y,T)$ is within $2\kappa$ from  $d^\ang_Z
           (X,T)$. Behrstock inequality gives that $d^\ang_Y(Z,T) \leq
           \kappa$ and therefore $d^\ang_Y (X,T) \geq d^\ang_Y (X,Z) -2\kappa$.
         \end{proof}

         \begin{lemma}{\it (Orbit estimates, or transfer in a coordinate)} \label{lem;transfert}

           Assume that $\Gamma_X$ has proper isotropy.

           For the  finite subset $F = F( 10\kappa)$ of $\Gamma_X$, and  for all
           $Y\in \Act(X)$,
           and all $X'$ that is either in  $ \Act(Y)$  or in $\Act(X)$, and all $\gamma\in
           \Gamma_X\setminus F$,  then either $d_Y^\ang
           (X',X)\leq \kappa$ or  $d_Y^\ang
           ( \gamma X',X)\leq \kappa$.
       \end{lemma}
       \begin{proof}
         Let us first treat the case of   $ X'\in \Act(Y)$. 
         If $d_Y^\ang
         (X',X)\leq \kappa $ we are done.  Assume that  $d_Y^\ang
         (X',X)> \kappa $.  
         By closeness in inaction, $  X'\in \Act(X)$,     
         and by Behrstock inequality (and because $\kappa \geq \theta$),
         one has  $d^\ang_X(X',Y) \leq \kappa$.  By
         proper isotropy (and coarse triangle inequality),   
         $d^\ang_X(\gamma X',Y) >
          5\kappa$. Thus, by Behrstock inequality again, $d_Y^\ang
         ( \gamma X',X)\leq \kappa$.

         Now assume that  $ X'\notin \Act(Y)$, but is in
         $\Act(X)$. Since $Y\in \Act(X)$ we can measure $d_X^\ang(X', Y)$
         and (since $\Gamma_X$ preserves $\Act(X)$) also
         $d_X^\ang(\gamma X', Y)$. By proper isotropy,  $d_X^\ang(X',
         \gamma X') \geq  10\kappa$ and therefore at least one of the
         quantities $d_X^\ang(X', Y)$ and $d_X^\ang(\gamma X', Y)$ is
         larger than $4\kappa$. Assume for instance that  $d_X^\ang(X', Y) \geq
         4\kappa$. Then by Behrstock inequality, $d_Y^\ang
         (X',X)\leq \kappa$. 
       \end{proof}

       To facilitate notations, we will say that a property is true for almost all elements of a group if the property holds for all elements outside a certain finite subset of the group.
       Using this lemma four times, together with triangle inequality, one
       gets:

         \begin{lemma} (Orbit estimates for proper isotropy)\label{lem;orbit_estimates}

           Let $X_1,X_2, X'_1, X'_2$ such that $X_1, X_2 \in
           \Act(Y)$. Assume that either $X'_i$ is in $\Act(Y)$ or in $\Act(X_i)$.

           If the group $\Gamma_{X_1}$ and $\Gamma_{X_2}$ have proper isotropy, 
           then for almost all elements $\gamma_1\in \G_{X_1}$ and  $\g_2 \in
           \G_{X_2}$, one has $$d^\ang_Y( \g_1 (X'_1),
           \g_2 (X'_2) )   -4\kappa  \leq d^\ang_Y(X_1,X_2) \leq
           d^\ang_Y( \g_1 (X'_1),  \g_2(X'_2) )   +4\kappa.$$
 
         \end{lemma}

         Recall that  we chose $K > 2\Theta +\kappa$.
        
         \begin{prop}(Ellipticity)

           Given $X \in \bbY_*$, and any $j\leq m$, the group $G_X$
           has an orbit in     $\calP_K(\bbY_j)$ of diameter at most $1$.  
         \end{prop}

         \begin{proof}
         If $j=i(X)$, and more generally, if $G_X$ fixes an element $Y\in
         \bbY_j$, it is obvious. Assume then that $\bbY_j \subset \Act(X)$.

         The group $G_X$
         preserves the set $\{Z \in \bbY_j,  \bbY^j_{K_0} (X, Z) =
         \emptyset\}$ for any 
         $K_0$ hence for $K_0 = (K-\kappa)/2\geq \Theta$.  
         Consider $Z_a, Z_b$ in this set, we claim that
         $\bbY^j_{K} (Z_a,Z_b)$ is empty. Assume $Y\in   \bbY^j_{K}
         (Z_a,Z_b)$.    
         Since $Y\in
         \Act(X)$ we can consider $d_Y^\ang (Z_a, X)$ and $d_Y^\ang
         (Z_b, X)$. By triangle inequality, $d_Y^\ang (Z_a, X) +d_Y^\ang
         (Z_b, X) \geq   d_Y^\ang (Z_a, Z_b) -\kappa  \geq
         K-\kappa$. Thus, one of them needs to be larger than $(K -
         \kappa)/2$ hence $Y$ is either in $ \bbY_{K_0} (X, Z_a)$ or
         in 
         $\bbY_{K_0} (X, Z_b)$, and this is a contradiction to our assumption.
         \end{proof}

         \begin{prop} (Induced orders)\label{prop;induced_order}

           Consider  $X, Z \in \bbY_*$, with $Z\in \Act(X)$.  Assume that
           $\G_X, \G_Z$ are infinite subgroups of $G_X, G_Z$ with proper isotropy.

           For all $i\leq m$,  for all $M \geq \Theta + 12\kappa$, the set
           $\bbY^i_{M} (X, Z)$ is finite, and  carries a partial order $\dot{<}$, that is
           given by the order of $ \bbY^i_{M -4\kappa} ( \g_X
           (X^i), \g_Z Z^i)$, for arbitrary $X^i$, $Z^i$,     
           in $\bbY_i$,
           and almost all $\g_X\in \G_X$ and $\g_Z\in \G_Z$.

         \end{prop}

         \begin{proof}
           Let us first check that the set is finite. We may assume
           that there are  $X^i\in
           \Act(X)\cap \bbY_i$ and  $Z^i\in \Act(Z)\cap \bbY_i$,
           otherwise    $\bbY^i_{M} (X, Z)$   is empty.   By Lemma
           \ref{lem;transfert}, there exists $\gamma_X\in \Gamma_X$,
           $\gamma_Z\in \Gamma_Z$ such that each $Y\in
           \bbY^i_{M} (X, Z)$ is in either one of the four sets
           $\bbY^i_{M-3\kappa} (\eta_X X^i, \eta_Z Z^i)$ for $\eta_X
           \in \{1, \gamma_X\}$ and $\eta_Z
           \in \{1, \gamma_Z\}$. The union of these four sets is
           finite by properness axiom.

           We now need to check that the order on  $ \bbY^i_{M -4\kappa} (\g_X
           (X^i), \g_Z Z^i)$ includes all  $\bbY^i_{M} (X,
           Z)$ and   does not depend on the choice of
           the  points $X^i, Z^i$. By  Lemma
           \ref{lem;orbit_estimates}, for
           arbitrary  choice of   
           points, and for any $Y \in
           \bbY^i _{M} (X, Z)$, there is a finite set of $\G_X$ and of
           $\G_Z$ such that for all elements $\g_X$, $\g_Z$  outside
           these finite sets,    
           $Y \in \bbY^i_{M -4\kappa} (\g_X
           X^i,  \g_Z Z^i)$ (the finite sets depend on the choice of
           $X^i, Z^i$   though).  Since $\bbY^i _{M} (X, Z)$
           is finite, we may find a finite set of $\G_X$ and $\G_Y$
           suitable for all of them.
           Thus, for almost all $\g_X, \g_Z$, 
           all  $ \bbY^i_{M -4\kappa} (\g_X
           (X^i), \g_Z( Z^i) )$ is ordered, and the order, once
           chosen the points $X^i, Z^i$, does not depend on $\g_X, \g_Z$. 

           Assume that for two different choices of  points $X^i,
           Z^i$, namely 
           $(X_a^i, Z_a^i)$ and $(X_b^i, Z_b^i )$,
           the orders are different, and take $Y_1, Y_2$ such that
           $Y_1\dot{<}_a Y_2$ for the first order, and
           $Y_2\dot{<}_b Y_1$ for the other.

           $Y_1\dot{<}_a Y_2$ means that   $d_{Y_1} (Y_2, \g_Z ( Z_a^i )
           ) \leq \kappa$. By the orbit estimate,  $d_{Y_1}^\ang (Y_2,  Z )
            \leq 5\kappa$ for suitable $\g_Z$. 

            $Y_2\dot{<}_b Y_1$ means that $d_{Y_1} (Y_2, \g_X ( X_b^i )
           ) \leq \kappa$, and by the orbit estimate,  $d_{Y_1}^\ang (Y_2,  X )
            \leq 5\kappa$.  Finally, by coarse triangular inequality,
            $d_{Y_1}^\ang (Z,  X)            \leq 11\kappa$,
            contradicting the assumption that  $Y_1$ is in $\bbY^i_{M} (X, Z)$.

         \end{proof}

         \subsection{Convexity}

    \begin{defi}(Convexity)

         Let $L>10\kappa$. We say that a subset $\calW \subset \bbY_*$ is
         $L$-convex if:    
           for all $i$, for all $X, Z \in \calW\cap \bbY_i$, for
           all $j$, the set    $\bbY^j_{L} ( X , Z)$ is a subset
           of $\calW$.

           Let now $\calL=(L(1), \dots, L(m))$ be a $m$-tuple of positive numbers.  
           A subset $\calW$ of $\bbY_*$ is said $\calL$-convex if for all $X,Z
           \in \calW$, of same coordinate $i(X) = i(Z)$, and for all $j$, the
           set $\bbY^j_{L(j)}(X,Z) $ is a subset $ \calW$. 

  \end{defi}

  Note that being $L$-convex, for  $L>0$ is equivalent to being
  $(L,  \dots, L)$-convex.

\begin{defi}
           Let $\calW\subset \bbY_*$ non-empty, and $R\in \bbY_*\setminus \calW$ for which $\Act(R) \cap \calW$
           is non-empty. Let $L\geq 10\kappa$. 
           Define $\bbY_L(\calW, R)$ as the set
           of $Y\in \bbY_*$ satisfying the following. 
           \begin{itemize}
             \item $Y\in \Act(R)$ 
             \item $Y\notin \calW$
             \item $\calW \cap \Act(R)\cap\Act(Y)$ is non-empty, and
               for all $X\in\calW \cap \Act(R)\cap\Act(Y)$, one has $Y\in
               \bbY^{i(Y)}_L(X, R)$.
           \end{itemize}
\end{defi}

           \begin{prop}\label{prop;intervals_are_finite}
 Assume that for all $X\in \calW$,  $\calW$ is invariant by
            an infinite group  $\Gamma_X$  of rotations around $X$, 
             with proper isotropy.
             If $L\geq \Theta + 12\kappa$,
             then for all $R$  for which it is defined, the set $\bbY_L(\calW, R)$ is finite.
           \end{prop}
           \begin{proof}
           From the definition, $\bbY_L(\calW, R) \subset
           \bigcup_i \bigcap_{X\in\Act(R)\cap \calW} ( \bbY^{i}_L(X,
           R) \cup (\bbY_i \setminus \Act(X)))
          $.   By finite filling assumption on the projection system, there 
          is a finite collection of elements $X_j\in \calW\cap
          \Act(R)$ such that $\cup_j
            \Act(X_j)$ covers  $\cup_{\calW\cap\Act(R)} \Act(X)$.

          In particular,  $\bbY_L(\calW, R)$ is inside a finite union of sets of
           the form 
           $\bbY^{i}_L(X_j, R)$ which are finite by Proposition
           \ref{prop;induced_order}. 
           \end{proof}
           \begin{prop}\label{prop;included}
             Assume that for all $X\in \calW$,  $\calW$ is invariant by
            an infinite group  $\Gamma_X$  of rotations around $X$, 
             with proper isotropy. Let $L\geq \Theta +12\kappa$. 
 
             If $\calW$ is $(L-6\kappa)$-convex,  
             and if $S\in  \bbY_L(\calW, R) $ then $\bbY_{L}(\calW, S)
             \subset \bbY_{L-2\kappa} (\calW, R)$. 

             Moreover, if $\calW'$ contains $\calW$, then  $\bbY_{L}(\calW', R)
             \subset \bbY_{L} (\calW, R)$.
           \end{prop}

           \begin{proof}
           Let $Y \in \bbY_L(\calW, S)$ in coordinate $i$. There
           exists $X\in \calW \cap \Act(Y)\cap \Act(S)$ such that $d^\ang_Y(X,S) \geq L$.

           Assume
           that $\tilde X \in \Act(R)\cap \Act(Y) \cap \calW$.  If it
           is not in $\Act(S)$, then $d^\ang_Y(\tilde X, S) <\kappa$
           and $d_Y^\ang(\tilde X, X) > L-2\kappa$. Transfering
           $\tilde X$ in the coordinate of $X$ (by invariance under
           $\Gamma_{\tilde X}$), one has 
           $d_Y^\ang(\tilde X_{i(X)}, X) > L-6\kappa$.  By convexity,
           $Y\in \calW$ though we assumed otherwise. 
           Therefore,  $\tilde X \in \Act(S)$. Therefore, by
           definition of $\bbY_L(\calW,S)$, one has $d_Y^\ang(\tilde
           X, S) \geq L$, but also $d_S^\ang(\tilde X, R) \geq
           L$.   It follows by transitivity of betweenness (Lemma
           \ref{lem;bet_trans}) that
           $d_Y^\ang(\tilde X, R) \geq L-2\kappa$.

           The second assertion is a direct consequence of the definition.
           \end{proof}

           \begin{prop} \label{prop;firstguy}
 Assume that for all $X\in \calW$,  $\calW$ is invariant by
            an infinite group  $\Gamma_X$  of rotations around $X$, 
             with proper isotropy.
             If $\Act(R)\cap\calW$ is not empty,
             for all $L\geq (2m+12)\kappa +\Theta$, there exists $Z\in
             \bbY_L(\calW, R)$ such that $\bbY_{L-2m\kappa} (\calW, Z)  = \emptyset$.
           \end{prop}

           \begin{proof}
           Let us say that $R$ has $k$  $L$-links to $\calW$ if $\{ i,
           \bbY_{L}(R,\calW) \cap \bbY_i \neq \emptyset \}$ has $k$
           elements.

           For any such index $i$, take a minimal item $Z_i$ in
           $\bbY_{L}(R,\calW) \cap \bbY_i $ for the order of
           Proposition \ref{prop;induced_order}. Then, by
           Proposition \ref{prop;included}, $\bbY_{L-2\kappa}
           (\calW, Z_i)$ is included in $\bbY_{L}(R,\calW) $, thus
           $Z_i$ has at most $(k-1)$  $(L-2\kappa)$-links to
           $\calW$. 

           Iterating this choice at most $m$ times, we find an 
           element $Z$ that has no $(L-2m\kappa)$-links to
           $\calW$. Therefore $\bbY_{L-2m\kappa} (\calW, Z)  = \emptyset$.
           \end{proof}

           \begin{prop}\label{prop;osc_cvx}

             Let $L\geq \Theta +12\kappa$. Consider $\calW$, and assume it  is 
             $L$-convex, and that for all $X\in \calW$, there is
             $\Gamma_X<G_X$, infinite, that leaves $\calW$ invariant
             and that has proper isotropy.
 
             If $\bbY_{L'}(\calW, R)$ is
             well defined and empty, then $\calW\cup\{R\}$ is $(L+L'+5\kappa)$-convex.
           \end{prop}

           \begin{proof}
             If $\calW\cap \bbY_{i(R)}$ is empty, there is nothing to
             prove. We assume it is non-empty.
           Consider $Y\in \bbY_{L+L'+5\kappa}(R, X)$ for some $X\in
           \calW\cap \bbY_{i(R)}$, and assume that $Y\notin
           \calW$. Notice that $Y\in \Act(R)$ though, and $X\in
           \Act(R)$ since they have same coordinate. 
            Hence,  $X\in \calW\cap\Act(R)\cap
           \Act(Y)$. 

           Let $X'$ be any other element of  $\calW\cap\Act(R)\cap
           \Act(Y)$.  Transfer $X'$ in the coordinate $i=i(R)$, inside
           $\calW$, by 
           $\Gamma_{X'}$. 
           There exists $X'_i \in \bbY_i \cap \calW$ such
           that $d_Y^\ang(X', X'_i) \leq \kappa$. But, $\calW$ being
           $L$-convex, one has $d_Y^\ang(X', X'_i) \leq  L$. It
           follows by triangular inequality, that
           $d_Y^{\ang}(R,X')\geq L'+2\kappa$. Since this is true for
           all $X'$ as above, it follows that $Y\in \bbY_{L'
             +2\kappa}(\calW, R)$, contradicting our assumption.

           \end{proof}

\section{Composite rotating families and windmills}

     We proceed to adapt the rotating families study of \cite{DGO} to
     the context of composite projection systems.

    \subsection{Definition}

    \begin{defi}(Composite rotating family)\label{def;CRF}

    A composite rotating family on a composite projection system,
    endowed with an action of a group $G$  by isomorphisms,   
    is a  family of subgroups $\G_Y,
    Y\in \bbY_*$ such that   
    \begin{itemize}
      \item for all $X\in \bbY_*, \G_X < G_X= {\rm
     Stab}_G(X)$,   is an infinite group of rotations around $X$, with proper isotropy   
      \item for all $g\in G$, and all $X\in \bbY_*$ , one
    has  $\G_{gX}= g\G_X g^{-1}$
    \item if $X\notin \Act(Z)$ then $\Gamma_X$ and $\Gamma_Z$ commute, 
      \item for all $i$, for all $X,Y, Z \in \bbY_i$, if $d_Y(X,Z)
        \leq \Theta_P $
        then for all
        $g\in \G_Y\setminus\{1\}$,  $d_Y(X,gZ) \geq \Theta_{Rot}$. 
    \end{itemize}

    \end{defi}

    We will show the following.

    \begin{theo} \label{theo;main}

      Consider $\bbY_*$ a composite projection system. 
      If $\{\Gamma_Y, Y\in \bbY_*\}$ is a composite rotating family
      for sufficiently large $\Theta_{Rot}$, then the group
      $\Gamma_{Rot}$ generated
      by $\bigcup_{Y\in \bbY_*} \Gamma_Y$,  
      has a partially commutative
      presentation. 

      More precisely, two presentations of $\Gamma_{Rot}$ are 
 $$  \Gamma_{rot} \simeq  \langle \bigcup_{Y\in \bbY_*} \Gamma_Y \;|\;  \begin{array}{cc} \forall Y, \forall Y'\notin \Act(Y)  &   [\Gamma_Y, \Gamma_{Y'}]=1    \\  \forall Y, \forall g\in \Gamma_{rot}  & \Gamma_{gY} = g\Gamma_Y g^{-1}   \end{array}    \rangle$$
and, for a certain  $\calS\subset
      \bbY_*$,
$$    \Gamma_{rot} \simeq  \langle \bigcup_{Y\in \calS} \Gamma_Y  \;|\;  \begin{array}{cc}    \forall Y,Y' \in \calS,  &  \\ \forall   w /\,  Y\notin \Act( wY'), &   [s,ws'w^{-1}]=1   \\   \forall s \in Y, s'\in Y'   &       \end{array}  \rangle     $$
  \end{theo}

In these presentations, we consider implicit the relations of the groups 
 $\Gamma_{Y}$ that appear in the generating sets. Moreover the expression  $\Gamma_{gY} = g\Gamma_Y g^{-1}$ refers to the following precise collection of formal relations: for all $\gamma$ in $\Gamma_Y$, for all $g \in \Gamma_{rot}  $, given the element $\gamma'\in \Gamma_{gY}$ equal to $ g\gamma g^{-1} $ (which exists by definition of composite rotating family), we add the relation $ (\gamma')^{-1}g\gamma g^{-1}=1 $. It is somewhat tautological, but necessary in a presentation over this generating set. The point of the second presentation is to avoid these tautological relations by reducing the generating set to a certain set of representatives of conjugacy classes of groups $\Gamma_Y$.

Unfortunately, it is not so easy to describe a-priori the subset $\calS$. It is constructed recursively in a number of steps, by taking at each step orbit representatives of a certain subset of $\bbY_*$ under the action of the group generated by the $\Gamma_Y$ that have been collected so far in the process. In principle, it probably can be enumerated explicitely, but at the cost of a certain complexification of the exposition.

The following result is, in our point of view, an incarnation of the
    Greendlinger lemma, from the small cancellation theories. If one
    considers a relation $\gamma$ of the quotient group, one can find
    in it a large part of a defining relation $\gamma_s$. Compare to
    \cite[\S 5.1.3]{DGO}. 

Let us consider $\Gamma_{rot} $ as in the previous theorem, and   $\gamma \in \Gamma_{rot}$. A principal coordinate for $\gamma$ is a coordinate $i\leq m$ for which, for all $X\in \bbY_{i}$,   $d_R(X, \gamma X) >\Theta_{Rot}    -2\Theta_P -\kappa$ (the constants are somewhat ad-hoc, chosen for the counting arguments to flow properly).         In that case, a shortening pair  $(R, \gamma_s)$   for  $\gamma$ in a principal coordinate $i$, at   $X\in  \bbY_{i}$,    is a pair consisting of a element $R$ of $\bbY_{i}$, and of an element $\gamma_s \in  \Gamma_{R}$ such that  %for all $X\in  \bbY_{i}$, 
  $d_R(X,  \gamma_s \gamma X) \leq 2\Theta_P +3\kappa$.

\begin{theo}\label{theo;mainGreendlingerLemma}
 Consider $\bbY_*$ a composite projection system. 
      If $\{\Gamma_Y, Y\in \bbY_*\}$ is a composite rotating family
      for sufficiently large $\Theta_{Rot}$, let  
      $\Gamma_{Rot}$ be the group generated
      by $\bigcup_{Y\in \bbY_*} \Gamma_Y$

      Then  
      for all $\gamma\in \Gamma_{Rot}\setminus \{1\}$, there
      is   $i(\gamma)\leq m$ a principal coordinate for $\gamma$ and 
      $(R, \gamma_s)$ a shortening pair for $\gamma$ in that coordinate.

    \end{theo}

    A major tool for analysing rotating families was the concept of
    windmills.  We are going to use composite windmills.

    Let us fix $\calL$ the $m$-tuple $$\calL =(  c_* +20(m-1)\kappa ,
    \,   c_*
    + 20(m-2)\kappa, \,  \dots, \,  c_* +20\kappa, \, c_*
    ).$$ Let $\sigma$ be the cyclic shift  on $\bbZ/m\bbZ$: $\sigma(i)
    = (i-1)$, and define $\calL_{j} = \sigma^{j-1} (\calL)$ obtained by
    shifting the coordinates of the $m$-tuple. 

    Thus  
    $\calL_i$ reaches its maximum $ c_*+20(m-1)\kappa$ on the coordinate
    $i$,  minimal value $c_*$ at $i-1$. Note that the maximum of
    $\calL$ is less than $\Theta_P -\kappa$.

    \begin{defi}(Composite windmills)\label{def;CW}

    A composite windmill is a collection $(\calW_1, \dots, \calW_m,
    G_W, j_0)$ in which 
    \begin{itemize}
      \item  $G_W$ is the subgroup of $G$ generated by  a set of
        subgroups $\{\G_Y, Y\in
        \bigcup_{i\in I_*} \calW_i\}$ for $I_*$ either $\{1, \dots,
        m\}$ or   $\{1, \dots,
        m\} \setminus \{j_0\}$,
      \item $\calW_i$ is a subset of $\bbY_i$ for all $i$, invariant
        under $G_W$,
      \item $j_0$ is called the principal coordinate, and $1\leq j_0\leq m$,  
      \item $ \bigcup_i \calW_i$ is   $\calL_{j_0}$-convex.
      \item The group $G_W$ has a partially commutative presentation,
        that is a presentation of the form $$ G\simeq \langle \calS \,
        |\, \calR\rangle$$ where $\calS$ is the union over a subset
        $\calW_*$ of
        $\calW$ of generating sets for $\Gamma_X, X\in\calW_*$, and
        $\calR$ consists of words over the alphabet
        $\calS \cup \calS^{-1} $ of the form $[s, ws'w^{-1}]$ for $w$ a word over
        $\calS \cup \calS^{-1}$. Moreover,  if $X,X' \in \calW_*$ and  $s\in
        \Gamma_X, s'\in \Gamma_{X'} $, the word  $[s, ws'w^{-1}] $ is in $\calR$ if and only if   $wX' \notin \Act(X)$. 
        \item (Greendlinger property) for each $\gamma \in
          G_W$ there is $i(\gamma)\leq m$, and for all $X\in
          \calW_{i(\gamma)}$, either $\gamma \in \Gamma_X$, or there is an $R\in
          \calW_{i(\gamma)}$ such that 
           $d_R(X, \gamma X) >\Theta_{Rot}
          -2\Theta_P -\kappa$. Moreover, there is a $\gamma_s\in \Gamma_{R}$ such
          that $d_R(X,  \gamma_s \gamma X) \leq 2\Theta_P +3\kappa$ (the pair
          $(R, \gamma_s)$ is called a shortening pair for
          $\gamma$ at $X$).

    \end{itemize}

    We say that the composite windmill has full group if $G_W$ is the subgroup of $G$ generated by $\{\G_Y, Y\in
        \bigcup_{i=1}^m \calW_i\}$. 
    \end{defi}

   If we do not mention it,  our windmills will be full. Only in
   specific circumstances do
   we need non-full windmills.  Indeed, we will use the case of a non-full group only at most one  time by
    coordinate, when
    initiating the process in each coordinate. 

    \begin{prop}
     In a composite windmill $\mathcal{W}$,  for all $i$ such that $\mathcal{W}_i \neq \emptyset$,  $\mathcal{W}_i$  is connected in   $\calP_K(\bbY_i)$. 
    \end{prop}

\begin{proof}    Consider $X, X'$ two points in it, by \cite[Thm. 3.7]{BBF} (more precisely the first claim in its proof), there exists a  path between them, $X_1, \dots X_n= X'$ such that for each $j$,  $X_j \in \bbY^i_K(X,X')$.  Since  $K>  \max (\calL) $, it follows that each $X_j$ is in  $\mathcal{W}_i$.
\end{proof}

    We say that a windmill $\calW'$ (with its  representative set $\calW'_*$
     used for the presentation of the definition)   is \emph{constructed over} $\calW$ if
    $\calW\subset \calW'$ and if the set of representatives $\calW'_*$
    contains the set of representatives $\calW_*$.  Note that it is transitive: if $\calW''$  is constructed over  $\calW'$, and  $\calW'$ is constructed over $\calW$, then  $\calW''$  is constructed over  $\calW$.

    \subsection{Osculations of two kinds}

       \begin{itemize}
         \item An osculator of type {\it gap} of a composite windmill $(\calW_1, \dots, \calW_m,
       G_W, j_0)$        
       is an element $R$ of $\bbY_{j_0} \setminus
       \calW_{j_0}$ such that there
       exists $i\leq m$,  $X_i, Z_i \in \calW_i$, that are in $\Act(R)$
       and such that $d_R^\ang(X_i,Z_i) > \frac{c_*}{2} -20\kappa$. 

         \item         An osculator of type {\it neighbor} of a
           composite windmill $(\calW_1, \dots, \calW_m,
       G_W, j_0)$         
       is an element
       $R$ of $\bbY_{j_0} \setminus \calW_{j_0}$ such that $\bbY_{\frac{c_*}{2}}(\calW, R) = \emptyset$.

       \end{itemize}

      \begin{lemma}  \label{lem;approx_gap} Consider a composite windmill $\calW = (\calW_{1}, \dots, \calW_m, 
  G_W, j_0)$, assume that $\calW_{j_0} \neq \emptyset$, and let $R \in \bbY_{j_0}$ be an osculator of type gap. 
  
Let $Y\in \bbY_i$ in $\Act(R)$.   Then there exists $X\in \calW_{j_0}$ such that
$d_Y^\ang(X,R) \leq \kappa$.
       \end{lemma}
\begin{proof}
      If $R$ is an osculator of type gap, there are
      $X', Z' \in \calW_i$, for some $i$, such that $d^\ang_{R}(X',Z')
      >c_*/2 -20\kappa$.

      Let  $X_0\in \calW_{j_0}$, and
      consider its orbit under the groups $\Gamma_{X'}$, and
      $\Gamma_{Z'}$,  which
      preserves $\calW_{j_0}$. 
      We may use Lemma \ref{lem;orbit_estimates} to  find
      $X'^{(j_0)}, Z'^{(j_0)}$ in these orbits, hence in $\calW_{j_0}$, such that $d^\ang_{R}
      (X'^{(j_0)}, Z'^{(j_0)}) >c_*/2-24\kappa$. 

      By the coarse
      triangle inequality, for at least one point among $X'^{(j_0)},
      Z'^{(j_0)}$, say  $X'^{(j_0)}$, we have $d^\ang_{R} (Y, X'^{(j_0)})
      >c_*/4 - 13\kappa$. Behrstock inequality gives   $d^\ang_{Y} (R,
      X'^{(j_0)}) \leq \kappa$.

\end{proof}

       \begin{lemma}\label{lem;i=1}
         Let $\calW$ be a composite windmill, and $R_1, R_3$ be two osculators
         of $\calW$.         Assume $\calW_{j_0} \neq \emptyset$, and 
         let $X_2 \in \calW_{j_0}$.

         If $R_3$ is of type neighbor and $\calW$ is $(\frac{c_*}{2}
         -20\kappa)$-convex, then $d_{R_1}(X_2, R_3) \leq  c_* $. 

         If $R_3$ is of type gap, then
         $d_{R_1}(X_2, R_3) \leq  \Theta_P  $. 
       \end{lemma}
       
       \begin{proof}
       If $R_3$ is an osculator of neighbor type, then the result
       follows from Proposition \ref{prop;osc_cvx}.

       If now $R_3$ is an osculator of type gap, the proof is slightly
       more involved. There is $i$, and there are $X,Z \in \calW_i$
       such that $ d_{R_3}^\ang (X,Z) > c_*/2-20\kappa$.  

       Since $\calW_{j_0}$ is non-empty, and invariant for $\G_{X}$ and
       $\G_Z$, we can apply Lemma \ref{lem;orbit_estimates}  and find
       $X^{(j_0)} , Z^{(j_0)} \in \calW_{j_0}$ such that $d_{R_3} (
       X^{(j_0)} , Z^{(j_0)} ) \geq  d_{R_3}^\ang (X,Z)  -4\kappa$
       which is $\geq  c_*/2-24\kappa$.  By coarse triangular
       inequality, at least one of the quantities $d_{R_3} ( R_1,
       X^{(j_0)} )$ and $ d_{R_3} (R_2, Z^{(j_0)})$ is greater than
       $c_*/4 -13\kappa$. Say it is $d_{R_3} ( R_1,
       X^{(j_0)} )$. Behrstock inequality then gives that $d_{R_1} ( R_3,
       X^{(j_0)} ) \leq \kappa$, and again coarse triangular
       inequality gives $d_{R_1} (
       X^{(j_0)} , X_2 ) \geq  d_{R_1}(X_2, R_3)      - \kappa$. Since
       the first is bounded by the maximal  convexity constant of
       $\calW$, the result follows.

       \end{proof}

       \subsection{The unfolding in the different coordinates}

       Given a composite windmill $\calW$, we will define its unfolding.

       Observe first the following, which justifies the next definition of admissible set of osculators.

       \begin{lemma} \label{lem;keepwalking}
         If $\calW$ is a composite windmill, it has some gap osculator if and only if it is not  $(\frac{c_*}{2}-20\kappa)$-convex. 

         Assume that for all $R\in \bbY_*$,  $\Act(R) \cap \calW \neq \emptyset$.   If $\calW$  is  $(\frac{c_*}{2}-20\kappa)$-convex, and yet does not contain $\bbY_*$, then there exists a neighbor osculator.
       \end{lemma}

       \begin{proof}
         The first assertion is direct from the definitions. To prove the second assertion, take $X \notin \calW$. By Proposition \ref{prop;firstguy} there is $Z$ in  $\bbY_{\frac{c_*}{2} + 2m\kappa}(\calW, X) \cup\{X\}$ such that  $\bbY_{\frac{c_*}{2} }(\calW, Z) = \emptyset$. It is therefore a neighbor osculator of $\calW$. 
 
\end{proof}

       We define now admissible sets of osculators of a composite windmill $\calW$  that does not cover the entire set $\bbY_*$.

        If  $\calW$ is not $(\frac{c_*}{2}-20\kappa)$-convex, then the (only) admissible set of osculators for $\calW$ is  the
       set $\calR_{gap}$ of osculators of type gap in $\bbY_{j_0}$. Note that it can be the empty set if the gap osculators are not in the coordinate $j_0$.

       If  $\calW$ is $(\frac{c_*}{2}-20\kappa)$-convex 
       (but does not cover the entire set $\bbY_*$),   then an admissible set of osculators for $\calW$ is a set   $\calR=\{  G_W R \} $ for  a choice of an osculator $R$  (necessarily of type neighbor).

       We  define the unfolding of $\calW$ as follows.

       \begin{defi}(Unfolding)\label{def;unfolding}  
         Let  $\calW = (\calW_1, \dots,
         \calW_m, G_W, j_0)$ 
         be a composite widmill that does not contain the entire set $\bbY_*$, 
         and $\calR$ be an admissible set of osculators.

         Define,
       for all $i$, $\calW'_{i}$ to be the union of all the
       images of $\calW_{i}$ by elements of the group $G_{W'}$ generated by
       $G_W \cup \{\displaystyle \bigcup_{\calR} \G_R\}$.   
       The unfolding
       of $\calW$ is then 
       $(\calW_1', \dots, \calW_m', G_{W'}, j_0+1)$, 
       where $j_0+1$ is taken modulo $m$.

       If $\calW$ contains  $\bbY_*$, its unfolding is $\calW' = \calW$.
       \end{defi}

       Here is an obvious lemma.
       \begin{lemma} (Trivial unfolding) \label{lem;emptyR}
         Let $\calR$ be a choice of an admissible set of osculators of
         $\calW$. If $\calR$ is empty, then the unfolding 
         $\calW' = (\calW_1, \dots,
         \calW_m, G_W, j_0+1)$ is a composite windmill. 
       \end{lemma}
 
       We thus concentrate on the case where $\calR$ is non-empty.

       In the case $\calW_{j_0}$ is empty, we include here a convexity result
       for an intermediate step in the construction: adding  an
       admissible set of osculators $\calR$,
       which produces a
       non-full composite windmill.

       \begin{lemma}\label{lem;calRisconvex}

         Assume that $\calW$ is a full composite windmill of principal
         coordinate $j_0$,  with $\calW_{j_0} = \emptyset$. 

         Let $\calW_{j_0}^s$ be a set $\calR$ of admissible osculators
         as defined above, assumed non-empty.

         For all other
         coordinates, let $\calW_i^s = \calW_i$. 

         Then $\calW^s= (\calW_1^s, \calW_2^s, \dots, \calW_m^s, G_W, j_0)$
         is a non-full composite windmill  of principal
         coordinate $j_0$. If moreover $\calR$ is the orbit of a
         neighbor osculator,  and if $\calW$ is 
         $(\frac{c_*}{2}-20\kappa)$-convex,  then $\calW^s$    is $B$-convex, for $B=  \frac{c_*}{2}+10\kappa \leq \inf \calL$.

       \end{lemma}

\begin{proof}
If $\calR=\emptyset$, there is nothing to prove.
Consider the case of the orbit of a neighbor osculator. 
It suffices to check that $\calW_{j_0}^s  (= G_W R)$ is convex in the
sense that for all $\gamma\in G_W$ and all $i$    
the set 
$\bbY^i_{B} (R, \gamma R)$ is in $\calW_i$.

By the Greendlinger Property,   
 given $\gamma$, there
exists $j$,  and $Y_j\in
\calW_j$ such that $d_{Y_j}(R, \gamma R) >\Theta_{Rot} -
  2\Theta_P -\kappa$, or $R=\gamma R$ (if $R$ is not active for all the shortening
pairs of $\gamma$).

Of course we consider only the first case of the alternative.

Assume that some $Y\in \bbY_i $  is in $\bbY_{B}^{(i)} (R,
\gamma R)$.  

If $Y\notin \Act(Y_i)$, then one can use a shortening pair at $Y_i$  to
reduce the length of $\gamma$ in its principal coordinate, and this
shortening pair gives $\gamma'$ such that $d_Y^\ang(R, \gamma R) =
d_Y^\ang(R, \gamma' R)$. Thus, $Y\in \bbY^{(i)}_{B} (R, \gamma' R)$
as well, and by performing this reduction sufficiently many times, we
may assume that $Y\in \Act(Y_i)$.

By Lemma \ref{lem;transfert}, either  $R$ or $\gamma R$ approximates by
$\kappa$ the projection of $Y_j$ on $Y$. 
 
 Say that $d^\ang _Y(\gamma R, Y_j) \leq \kappa$. By osculation 
 if $Y\notin \calW_i$, one has  $d_Y^\ang(Y_j, R) \leq \frac{c_*}{2}
 $. Therefore, 
 one has $d_Y^\ang(\gamma R, R)\leq  d^\ang _Y(\gamma R, Y_j) +d^\ang
 _Y( R, Y_j)  +\kappa \leq  \frac{c_*}{2} +2\kappa $ which is less than $B$.

If now $d^\ang _Y( R, Y_j) \leq \kappa$, one has $d_Y^\ang (R,\gamma
R) $ is within $2\kappa$ from $d_Y^\ang (Y_j ,\gamma
R) $, which equals $d_{\gamma^{-1}Y}^\ang (\gamma^{-1} Y_j ,
R) $. Of course, $Y\notin \calW_i$ if and only if $\gamma^{-1}Y \notin
\calW_i$, hence, if it is the case, by osculation of $R$,
$d_Y^\ang(Y_j, \gamma R) \leq \frac{c_*}{2}$, 
and  $d_Y^\ang(\gamma R, R)\leq  \frac{c_*}{2} +
 2\kappa \leq B$.

In the case where $\calR$ is the set of gap osculators, the proof is
similar.  Indeed, if $R_1$ is a gap between $X_1$ and $Z_1$, and
$R_2$ is a gap between $X_2$ and $Z_2$, and  if $Y$ is
between $R_1$ and $ R_2$, so that $d_Y^\ang(R_1,
 R_2) \geq c_*+20(m-1)\kappa (=\calL_{j_0}(j_0))$,  then $Y$ is also between $X_1$ (or $Y_1$) and
$X_2$ (or $Y_2$) so that, say, $d_Y^\ang(X_1,
 X_2) \geq c_*+20(m-1)\kappa -3\kappa$. One can transfer $X_2$ in the
 coordinate of $X_1$ by Lemma \ref{lem;transfert}, in $\calW$ (in the
 $\Gamma_{X_2} $-orbit of $X_1$). The convexity of $\calW$
then shows that $Y\in \calW$.  

\end{proof}

The aim of the next sections is to prove the following.

\begin{prop}  \label{prop;unfold_main} If   $\calW= (\calW_1, \dots, \calW_m, G_{W'}, j_0)$ is a
  (full) composite windmill, and  $\calR$ is an admissible set of osculators, then the unfolding  $\calW'= (\calW'_1, \dots, \calW'_m, G_{W'}, j_0+1)$ is a
  (full) composite windmill, and $\calW'_*$ can be chosen to contain
  $\calW_*$ (in other words, $\calW'$ is constructed over $\calW$).
\end{prop}

       \subsubsection{Unfolding a tree}

       \begin{prop} (Principal coordinate tree) \label{prop;principal_tree}

         Consider a full composite windmill $\calW$, of
         principal coordinate $j_0$.

         Let $\calR\neq \emptyset$ be an admissible set of osculators as defined in
         the previous section.  If  $\calW_{j_0} = \emptyset$, let $
         \calW^s_{j_0}=\calR$, and otherwise let
         $\calW^s_{j_0}=\calW_{j_0} $.

         There exists a $G_{W'}$-tree $T$, bipartite, with black
         and white vertices,  with an  equivariant
         injective map $\psi : T\to \calP(\bbY_{j_0})$ (the set of
         subsets of $\bbY_{j_0}$)  that sends  black vertices to images
         of osculators  by $G_{W'}$, and white vertices to
         images of $\calW^s_{j_0}$ by $G_{W'}$, and that sends the
         neighbors (in $T$) of the preimage of $ \calW^s_{j_0}$ to
         $\calR$.

         Moreover, for any pair of distinct white vertices $w_1, w_2$,
         and any black vertex $v$ in the interval between them (in $T$), and
         any $X_1\in \psi(w_1), X_2\in \psi(w_2)$,   one has
         $d_{\psi(v)}(X_1,X_2) \geq \Theta_{Rot} - 2\Theta_P -\kappa$.

         Finally, if $w_1, w_2$ are white vertices for which the path
         from a black vertex $v$ starts by the same edge, then for any 
         $X_1\in \psi(w_1), X_2\in \psi(w_2)$,   one has
         $d_{\psi(v)}(X_1,X_2)   \leq 2\Theta_P +3\kappa$. 
       \end{prop}

       \begin{proof}
       
      Take a transversal $\calR^t$ of $\calR$ under
       the action of $G_W$.   For each $R\in \calR^t$, let $(G_W)_R$
       the subgroup of $G_W$ generated by  $\bigcup_{X \in \calW \setminus \Act(R)}
       \Gamma_X$.

       Set $T$ to be the Bass-Serre tree of the
       (abstract) graph of groups whose vertex groups are $G_W$ and
       the groups $\G_R \times (G_W)_R  ,
       R\in \calR^t$, and the edges are the pairs $(G_W, R), R\in
       \calR^t$, and the edge groups are the groups $(G_W)_R$. 

       Let $\widetilde{G_{W'}}$ the fundamental group of this graph of
       groups. The group $G_{W'}$ is a quotient of this group, since
       it is generated by $G_W$ and the stabilizers of elements $R$ of
       $\calR^t$, which, by assumption (Definition \ref{def;CRF}), are  
       direct sums of their rotation group
       with the  groups $(G_W)_R$.

       $T$ is
       a tree, endowed with a $\widetilde{ G_{W'}}$-action, bipartite, 
       and with an  equivariant (with respect to $\widetilde{G_{W'}}
       \onto G_{W'}$) 
          map $\psi\,:\, T\to \calP(\bbY_j)$ that sends  black vertices to images
         of elements of $\calR$  by $G_{W'}$, and white vertices to
         images of $\calW^s_{j_0}$  
         by $G_{W'}$.

         We need to show that it is injective, and at the same time,
         we will show the estimate of the end of the statement.

         Consider a path $p$
         of
         $T$, starting and ending at  white vertex. Up to cyclic permutation, and up to
         the group action, we may assume that the path $p$ starts at
         the vertex fixed by $G_W$, and its second vertex is fixed by
         some $R_1\in \calR^t$, and that its length is even. 

         Let us denote by $p_0, p_1, \dots , p_N$ the consecutive
         vertices of $p$, and let $X_{2i}$ be a choice of a element of
         $\psi(p_{2i})$, and $R_{2i+1} = \psi(p_{2i+1})$.

         The monotonicity property in the coordinate $j_0$ says that
         if $d_Y(X,Z) \geq \Theta$ then $d_W(X,Z) \geq d_W(X,Y)$. 

         We
         will use it in an induction  
         to establish that  for all $k$ odd, and all $i$ in $1 \leq i\leq
         \frac{N-k}{2}$ and all $j$ in $1\leq j\leq \frac{k-1}{2} $,
         one has $$ \begin{array}{cccc} & d_{R_k}(R_{k-2j}, R_{k+2i}) &\geq & \Theta_{rot}
         -2\Theta_P -\kappa  \\ 
\forall X_{ s}\in  \psi(p_{s}), &        d_{R_k}(X_{k-2j+1},
                                  X_{k+2i-1}) & \geq  & \Theta_{rot}
         -2\Theta_P -\kappa \end{array}$$  

         The case $i,j=1$ happens as follows.  Choose $k$. 
         
         We first show how a black vertex separates two adjacent white
         vertices. Note that there is $X'_{k+1} \in \psi(p_{k+1})$ that equals
         $g X_{k-1}$ for some $g \in
         \G_{R_{k}}\setminus \{0\}$.  By  convexity of
         $\calW^s_{j_0}$ (ensured by assumption, or by 
         Lemma \ref{lem;calRisconvex} in case $\calW_{j_0}$ is empty), 
         $d_{R_{k}} (X_{k+1}, X'_{k+1}) \leq
         \Theta_P$.  And by assumption on the rotating groups,
         $d_{R_{k}} ( X_{k-1}, X'_{k+1}) \geq
         \Theta_{Rot}$. Thus,    $d_{R_{k}} ( X_{k-1}, X_{k+1}) \geq
         \Theta_{Rot}-\Theta_P-\kappa$, the second inequality.

         By Lemma \ref{lem;i=1},   $d_{R_{k}}(X_{k+1},  R_{k+2}) \leq
         \Theta_P$ and $d_{R_{k}}(X_{k-1},  R_{k-2}) \leq
         \Theta_P$. By triangle inequality, we get $
         d_{R_k}(R_{k-2}, R_{k+2}) \geq  \Theta_{rot}
         -2\Theta_P-\kappa$.  We have both inequalities.

         Assume that the inequalities are proven for all $(i, j)$ such that
         $i+j \leq i_0$  (and for all
         $k$), and  let us choose $k$ and $(i, j)$ with $i+j\leq i_0$, and prove the inequality  
         for $(i+1, j)$.

         Set  $Y= R_{k+2i}$, and  and
         $W = R_{k}$.    In the following we set either $Z=
         R_{2i+k+2}$ or $X_{2i+k+1}$,      and  either $X=R_{k-2j}$ or
         $X = X_{k-2j+1}$.

         By the inductive
         assumption for $k'=k+2i$, $i'=1, j'= i $, one has $d_Y(W ,
         Z)\geq  \Theta_{rot} -2\Theta_P-\kappa$. 

         Also for  $k$, $i$ and $j$  the induction gives    $d_{W}(Y,X)
         \geq  \Theta_{rot} -2\Theta_P-\kappa$.  Behrstock inequality
         then provides  $d_{Y}(W,X) \leq \kappa$ and therefore  $d_Y(X ,
         Z)\geq  \Theta_{rot} -2\Theta_P-3\kappa$. This is still far
         above $\Theta$. One thus may apply the monotonicity property
         and obtain  $d_W(X,Z) \geq d_W(X,Y)$. In other words, $$
         d_{R_k}(R_{k-2j}, R_{k+2i+2}) \geq  \Theta_{rot}
         -2\Theta_P-\kappa. $$

         The inequality is also proven for $(i, j +1)$ in the same
         manner, symmetrically. This finishes the induction.

         In the end, we have obtained for $i=N/2-1$, and $k=1$,  
         $ d_{R_1} (X_{0}, R_{N-1}) \geq
         \Theta_{Rot} -\Theta_P$, and it follows that  $ d_{R_1} (X_{0},
         X_{N}) \geq  \Theta_{Rot} -2\times \Theta_P-\kappa$, which is the
         estimate of the statement.  

         If we assume that $p$
         is mapped to a loop,  $\calW_{j_0}$ contains both $X_0$ and
         $X_N$, and not $R_1$ (it is an osculator),  the convexity of
         $\calW_{j_0}$  imposes  $\Theta_{Rot} -2\times \Theta_P-\kappa \leq
         \Theta_P$, meaning $\Theta_{Rot} \leq \Theta_P+\kappa$.
         and this contradicts our choice of $\Theta_{Rot}$. 

         It also follows from this analysis that if $w_1, w_2$ are
         white vertices of $T$ and $v$ is a black vertex between then,
         then $d_{\psi(v)} ( X_1, X_2 ) \geq \Theta_{Rot}- 2\Theta_P
         -\kappa$ (in our induction above). A final use of
         Behrstock inequality provides that whenever the paths from
         $v$ to a white vertex $w_1$   has more than three edges, then if $v'$ is the
          first black vertex after $v$ on this path, and  if $X_1\in
          \psi(w_1)$, then         
         $d_{\psi(v)}(X_1, \psi(v')) \leq \kappa$. It follows from
         that and Lemma \ref{lem;i=1} that if
         $w_2$ is another white vertex $w_1$  whose path from $v$
         starts at the same edge, $d_{\psi(v)}(X_1, \psi(v')) \leq
         2\Theta_P + 3\kappa $.

         \end{proof}

         The former proposition allows to define, for each element $\gamma$ of
         $G_{W'}$, its principal coordinate, and its principal
         tree. Indeed, if  $\gamma\in G_{W'}$ is not conjugated to $G_W$, the
         proposition shows that it is either loxodromic or the
         stabilizer of a black vertex on the tree $T$. Then
         we define its principal coordinate as $j_0$ and its principal
         tree as $T$. If it is in   $G_W$, or conjugate in it, its principal coordinate
         and its principal tree are defined inductively, according to
         the process of unfoldings of composite windmills.

       \subsubsection{Preservation of convexity}

       \begin{prop}(Convexity of $\calW'$)\label{prop;unfolding_is_convex}

         Let $\calW = (\calW_1, \dots, \calW_m, G_W, j_0)$ be a
         composite windmill (possibly non-full).   

         Assume that $\calR$ is an admissible set  of 
         osculators, and   $\calW'$ the unfolding defined in 
         Definition \ref{def;unfolding}

         If $\calR$ consists of the orbit of a neighbor,  then $\calW'$
         is $c_*$-convex.

         If $\calR$ consists of gap osculators, then $\calW'$ is
         $\calL_{j_0+1}$-convex.

       \end{prop}

       The case of $\calR=\emptyset$ is trivial, so we assume
       it is not empty.

       \begin{proof}

       If $\calR$ consists of the orbit of a neighbor, let  $A_j=
       c_*$ for all $j$.  If $\calR$
       consists of gaps, let  $A_j= \calL_{j_0}(j) + 20\kappa$ (which is
       less than $
       \calL_{j_0}(j+1)$). 

       Let $X,Z \in \calW'_i$, consider $Y\in \bbY_{A(j)}^{j} (X,Z)
       $.

       Here is our main claim. 
       
       We will show that $Y$ is a $G_{W'}$-translate of one of the
       following type of elements:
       \begin{itemize}
         \item $Y'$ for which there exists $ X_f, Z_f \in \calW_{j_0}$ such
           that $d_{Y'}^\ang(X_f, Z_f) \geq   A(j) - 10\kappa$; 
         \item $Y'$ for which there exists $ X_f \in \calW_{j_0}$, and $R$
           an osculator of $\calW$ in $\calW'_{j_0}$ such that $d_{Y'}^\ang(X_f, R) \geq   A(j) - 10\kappa$; 
         \item $Y'$ for which  there exists $R_1, R_2$ osculators of
           $\calW$  in $\calW'_{j_0}$ such that $d_{Y'}^\ang(R_1, R_2) \geq   A(j) - 10\kappa$ 
       \end{itemize}

       We will then finish the proof with this claim established, but
       before that we will prove the claim.

\newcommand{\Xjo}{X^{(j_0)}}
\newcommand{\Zjo}{Z^{(j_0)}}
\newcommand{\Yjo}{Y^{(j_0)}}

       {\it Transfer of $X$ and $Z$ to $\bbY_{j_0}$. } In $\calW'$,
       the groups $\Gamma_X$ and $\Gamma_Z$ preserve
       $\calW_{j_0}'$ which is not empty (it contains $\calR$). 
       Therefore, by Lemma \ref{lem;orbit_estimates} there are
       $\Xjo, \Zjo$ in $\calW'_{j_0}$ such that $d_Y^{\ang} (\Xjo,
       \Zjo) \geq A(j) - 4\kappa$.

       {\it The interval in $T$.}  Taking $\psi^{-1}$ of $\Xjo$ and of
       $\Zjo$ produces two vertices in the principal coordinate tree
       $T$ of Proposition \ref{prop;principal_tree}. More precisely, either one of $\Xjo,
       \Zjo$ is the image of a black vertex of $T$, or in the image of
       a white vertex of $T$. This thus give two vertices of $T$ that
       we (slightly abusively) denote   by  $\psi^{-1}(\Xjo),
       \psi^{-1}(\Zjo)$. 

       If these vertices are adjacent, we have achieved the second
       point of the 
       claim. If these vertices are the same, we have achieved the
       first point of the claim. If these vertices are different,  both black with
       only one white vertex in the interval, we have achieved the
       third point of the claim. 

       Thus, we may assume that there is at least one black vertex of
       $T$ in the open interval  $(\psi^{-1}(\Xjo),
       \psi^{-1}(\Zjo))$.  Let $R_1, \dots ,  R_N$ the images by
       $\psi$ of these black
       vertices, in order starting from the side of $\psi^{-1}(\Xjo)$.

       By Proposition \ref{prop;principal_tree}, we have for all $i$, $d_{R_i}(\Xjo, \Zjo)
       >\Theta_{Rot} - 2\Theta_P - \kappa$, which is
       $>50\kappa$.   
       
       {\it Reduction to the case where $R_i\in \Act(Y)$}

      If $Y$ is
      equal to one of the $R_i$ then we fall in the first possibility
      of the main claim.      Thus, let us assume
      that $Y$ is different from all the $R_i$. 
      
      We may assume that $Y$ is in $\Act(R_i)$  for all $i$. Indeed if it
      was not, one could use an element of $\Gamma_{R_i}$ to reduce
      the length of the path $p$, without changing the value of the
      projection distance $d^\ang_Y(X^{(j_0)}, Z^{(j_0)})$ since $ \Gamma_{R_i}$ leaves
      $d^\pi_Y$ invariant.

      {\it Transfer of $Y$ in $\bbY_{j_0}$}. We may apply Lemma
      \ref{lem;transfert} again, and find an element $\Yjo$ in $\bbY_{j_0}$ (far in an
      orbit of $\Gamma_Y$) such that, for all $i$,  one has
      $d_{R_i}^\ang ( Y, \Yjo)\leq 4\kappa$. 

      {\it Position of $\Yjo$ in the order}. 
      Fix $0<i\leq N$. Since $d_{R_i}(\Xjo, \Zjo) >50\kappa$, either
      $d_{R_i} (\Xjo, \Yjo)$ or   $d_{R_i} (\Yjo, \Zjo)$ is larger
      than $24\kappa$. 

      All $R_i$ are in $\bbY_{50\kappa}(\Xjo, \Yjo)$ therefore they
      satisfy the order property in this set, which coincide with the
      ordering of their indices. 
      By this order property and Behrstock
      inequality, if for some $i$ one has $d_{R_i}(\Yjo, \Xjo)
      >5\kappa$, then for all $i'<i$, one still has  $d_{R_i}(\Yjo, \Xjo)
      >5\kappa$.  Similarly if  $d_{R_i}(\Yjo, \Zjo)
      >5\kappa$ then for all greater $i''$ the same holds. 

       Therefore we have three cases.

       Either $d_{R_1}(\Yjo, \Xjo) \leq 5\kappa$ or $d_{R_N}(\Yjo,
       \Zjo) \leq 5\kappa$, or there exists $i\geq 1$, largest  such
       that $ d_{R_1}(\Yjo, \Xjo) > 5\kappa$ and $i<N$.

       By symmetry, and translation by an element of $G_{W'}$ the first
       and second case have same resolution. Let us treat the first
       one.  By triangle inequality, $d_{R_1}(\Zjo, \Yjo)
       >\Theta_{Rot} - 10\kappa - 2\Theta_P$ which is still greater than
       $20\kappa$.

       Going back to $Y$: $d_{R_1}^\ang (\Zjo, Y) >16\kappa$. By
       Behrstock inequality, $d_{Y}^\ang (\Zjo, R_1) <\kappa$, and
       finally by triangle inequality, $d_Y^\ang(\Xjo, R_1) \geq A(j)
       -2\kappa$. We are in the second point of the claim if $\Xjo$ is in a white
       vertex, and in the third point if it is a black vertex.

       We thus turn to the case where  there exists $i\geq 1$, largest  such
       that $ d_{R_1}(\Yjo, \Xjo) > 5\kappa$ and $i<N$.

       One has 
$$\begin{array}{ccl} 
d_{R_{i+1}} (\Yjo, \Zjo)  & > & \Theta_{Rot} -   2\Theta_P -10    \kappa \\ 
  d_{R_{i+1}}^\ang (Y, \Zjo)     & >  &  \Theta_{Rot} - 2\Theta_P -14   \kappa \\
d_Y^\ang(R_{i+1}, \Zjo) & \leq & \kappa \\
\end{array}
$$
and 
$$\begin{array}{ccl} 
d_{R_{i}} (\Yjo, \Xjo)  & \geq & 5    \kappa \\ 
  d_{R_{i}}^\ang (Y, \Xjo)     & \geq  &    \kappa \\
d_Y^\ang(R_{i}, \Xjo) & \leq & \kappa \\
\end{array}
$$

So, $d_Y^\ang (R_i, R_{i+1} ) \geq A(j)-4\kappa $. We have the third
point of the claim, and the claim is established.

We need to finish the proof of the lemma.  There are several cases to
treat. The easiest is when the first case of the claim occurs. 

In that case, if $j=j_0$, $Y'$ is actually a gap osculator, hence in
$\calW'_{j_0}$. If $j\neq j_0$, by convexity of $\calW$, it is in
$\calW_j$.  

Assume now that the second case occurs.

 If $R$ is of type neighbor, it simply contradicts Proposition
 \ref{prop;osc_cvx}.

 If $R$ is an osculator of type gap between $X_0, X_1$, and $j=j_0$,
 one easily gets that $R$ is an osculator of type gap between $X_f$
 and either $X_0$ or $X_1$ (any one for which $ d_R(Y',X_\epsilon)$ is larger
 than $\kappa$, and by triangular inequality, there must be at least
 one).  If $j\neq j_0$, we may use the same argument. $Y'\in\Act(R)$
 therefore   $d^\ang_R(Y',X_\epsilon)$ is larger
 than $\kappa$ for either $\epsilon=0$ or $1$. Then,
 $d^\ang_{Y'}(R,X_\epsilon) <\kappa$ and by triangular inequality,
 $d^\ang_{Y'}(X_f ,X_\epsilon) \geq A(j) -12\kappa (>\calL_{j_0}(j)) $.  It follows by
 convexity of $\calW$ that $Y'\in \calW_j$.

Finally, assume that the third case occurs.

Assume that $R_2$ is an osculator  of type  gap, between $X_0,X_1$. 
Then, again with the same
reasoning, $Y'\in \Act(R_2)$ and there is $\epsilon$ for which it is
in $\Act(X_\epsilon)$  and $d_{Y'}^\ang (R_2, X_\epsilon)$ is less
than $\kappa$. Thus $d_{Y'}^\ang (R_1, X_\epsilon) \geq A(j)-12\kappa$,
 and we are back to the case $2$ of the claim, with a slightly lower
 constant. The proof goes
 nevertheless through, and the desired conclusion holds.

Finally,       
 assume that $R_2$ is of type neighbor. Then   both $R_1, R_2$ are of type neighbor, and $R_2=
\gamma R_1$ for some $\gamma \in \G_W$. Let us rename $R_1=R$, call
$i=i(Y')$, and $j$ the principal coordinate of $\gamma$ (for the
Greendlinger property). Let $Z \in \calW_j$ be the vertex of a
shortening pair for $\gamma$ for which $Z\in \Act(Y') \cap
\Act(R)$ (there exists one, otherwise one can reduce the length of
$\gamma$ in its principal tree  by a
shortening pair at $Z$). Thus, $d_Z^\ang(R, \gamma R) >\Theta_{Rot} -
2\Theta_P -2\kappa$. 

Suppose $d_{Y'}^\ang(R, \gamma R) >  c_*-10\kappa$. Then, there are two possible
cases. Either $d_{Y'}^\ang(R,Z) > \frac{c_*}{2}- 6\kappa $ or $d_{Y'}^\ang(\gamma R,Z)
> \frac{c_*}{2} - 6\kappa$ (or both).

 In
the first case, $d_Z^\ang(R,Y') \leq \kappa$. Thus $d_Z^\ang(Y',\gamma
R) >\kappa$, and so $d_{Y'}^\ang(\gamma R, Z)< \kappa$. 

Recall that $Z\in\Act(R)\cap \Act(Y')$. Thus $d_{Y'}^\ang(Z,R) > c_*
-2\kappa$, and $Y'\in \bbY_{c_*-2\kappa}( Z,R  ) $. Now let $Z'$ any
other element of $\calW$ in  $\Act(R)\cap \Act(Y')$.  By $(\frac{c_*}{2}-20\kappa)$-convexity of
$\calW$,  one has  $d^\ang_{Y'} (Z,Z') \leq  \frac{c_*}{2}-20\kappa$
and therefore  
  $Y'\in
\bbY_{c_*-2\kappa  -\frac{c_*}{2}+21\kappa}( Z',R  ) $. In other words, $Y'
\in \bbY_{\frac{c_*}{2}+19\kappa}(\calW, R)$ and this contradicts the fact that $R$ is a neighbor.

In the second case, the situation is similar after composing by the automorphism
$\gamma^{-1}$.

\end{proof}

      \subsubsection{The unfolding is a windmill}

      \begin{prop} If $\calW= (\calW_1, \dots, \calW_m, G_W, j_0)$ is
        a composite windmill, and if $\calW'=  (\calW'_1, \dots, \calW'_m,
        G_{W'}, j_0+1)$  is an unfolding over an admissible set of osculators, 
        then $\calW'$ is a composite windmill.

        Moreover, the set $\calW'_*$ of the fifth point of the
        definition can be assumed to contain the set $\calW_*$ (in
        other words, $\calW'$ is constructed over $\calW$). 
      \end{prop}
 \begin{proof}
      The first three points follow by construction. The fourth point (convexity) is
      the result of Proposition \ref{prop;unfolding_is_convex}. The
      sixth point is a consequence of Proposition
      \ref{prop;principal_tree}.  The same proposition introduces an
      action of $G_{W'}$ on a tree $T$ which is Bass-Serre dual to  a
      presentation  of $G_{W'}$ as     the fundamental
      group of a graph of group, with one vertex $v_0$ carrying the group
      $G_W$ and the other vertices $v_{[R]}, [R]\in \calR/G_W$,
      adjacent to a single edge whose other end is
      $v_0$,  carrying the group $ \Gamma_R \times (G_W)_R $, if $R$
      is a representative of the orbit $[R]$.

\end{proof}

       \subsection{Towers of windmills, and accessibility}

       \subsubsection{Starting point}

       We start the process by selecting $\calW(0)$ to be  a maximal collection of
       mutually inactive elements in $\bbY_*$. Thus, whenever
       $\calW(0)_j \neq \emptyset$, it is reduced to a single point.

       We choose $j_0=1$. It is clear that $\calW(0)$ defines a composite windmill
       where for all $i$,  $\calW(0)_{i}$ is either empty or a singleton, and
       where $G_{W}$ is the direct product of the groups 
       $G_X$, for $X\in \calW(0)$ (there are at most $m$ direct
       factors).

       $\calW(0)$ is $\kappa$-convex,  and for all $R$, by maximality
       of $\calW(0)$, $\Act(R)\cap \calW(0) \neq \emptyset$.
       Recall that by choice, $c_*>25\kappa + 2\Theta$, hence by
       Proposition \ref{prop;firstguy},  there exists a neighbor
       osculator in $\bbY_{\frac{c_*}{2}+2m\kappa} (\calW(0), R)$.

       \subsubsection{The process}

       Recall that we assumed $\bbY_*$ to be
       countable.

       We will work with indices in the {\it set} of countable ordinals:
       we will define $\calW(k)$ for $k$ any countable ordinal (not
       necessarily a number). 
       We take the notation $$  \calW(k) = (\calW(k)_1, \dots, \calW(k)_m, G_{W(k)}, j_k).  $$ 
       Let us convene that  $\calW(k)\subset \calW(k')$ means that  for all $i\leq m, \calW(k)_i\subset \calW(k')_i$. This is not an order relation, however note that, for full windmills, if $\calW(k)\subset \calW(k') \subset \calW(k)$,  and if $\calW(k)$ is fixed, there are only $m$ possibilities for  $\calW(k')$ (corresponding to the values of $j_{k'}$).      
We will also write 
$\calW(k)\subsetneqq \calW(k')$ if $\calW(k)\subset \calW(k')$ and one of the inclusions $\calW(k)_i\subset \calW(k')_i$ is strict.

       We have chosen $\calW(0)$.  In order to define  $\calW(k)$ for $k$ any countable ordinal, we treat separately the case of $k$ a successor of some ordinal, and the case of $k$ a limit ordinal.

       For any countable ordinal $k$, we define  $\calW(k+1)$ to be the unfolding of $\calW(k)$  (as in Definition \ref{def;unfolding}) over an admissible set of osculators. Recall that if there is no gap osculator at all, one may need to choose a certain neighbor osculator to define a choice of admissible set of osculators. We could, but do not impose the choice.

Note that by  maximality of $\calW(0)$, Lemma \ref{lem;keepwalking} can be applied to show that such a choice is always possible for all $\calW(k)$.

      \begin{lemma}\label{lem;still_a_CWM}
       If  $\calW(k)$ is a composite windmill, then   
       $\calW(k+1)$ 
             is still a composite
       windmill, constructed over $\calW(k)$.

       \end{lemma}

\begin{proof} This follows from Proposition \ref{prop;unfold_main}
 if the set of
osculators is non-empty, and from Lemma \ref{lem;emptyR} otherwise.  
\end{proof}

We now define $\calW(\alpha)$ assuming $\alpha$ is a limit ordinal, and that all  $\calW(k)$, for $k<\alpha$ have been defined, and satisfy $ \calW(k) \subset \calW(k')  $ for all $k<k'$.

We  consider $\calW(\alpha)_i = \bigcup_{k<\alpha} \calW(k)_i$ for each $i\leq m$,  and  $ G_{W(\alpha)} = \bigcup_{k<\alpha} G_{W(k)} $,  and we set $j_\alpha = 1$.

\begin{lemma}\label{lem;still_a_CWM_again}
 If $\alpha$ is a limit  countable ordinal such that for all $k<\alpha$, $\calW(k)$ is a composite windmill, and that for all $k<\alpha$,   $\calW(k+1)$ is constructed over $\calW(k)$.  Then  $\calW(\alpha)$ is a composite windmill, constructed over $\calW(k)$, for all  $k<\alpha$.
\end{lemma}

\begin{proof}
One easily check that all the points, except possibly the fifth (on the partially commutative presentation)  of the definition \ref{def;CW} of composite windmill are satisied after taking a direct union. Assume that the fifth point is not satisfied. Consider then $\alpha_0$ the smallest ordinal such that this point fails.  $\alpha_0$ is  a limit ordinal (otherwise
 Lemma \ref{lem;still_a_CWM} says that  $\calW(\alpha_0)$ is a composite windmill constructed over earlier  $\calW(k)$).  Fix $k_0<\alpha_0$. For all $k<k_0$,  $\calW(k)$ is contained in $\calW(k_0)$.

Note that by definition,  for each $i\leq m$,   $$\calW(\alpha_0)_i = \bigcup_{k_0 <k<\alpha_0} \calW(k)_i, \hbox{ and } \,  G_{W(\alpha_0)} = \bigcup_{k_0< k<\alpha_0} G_{W(k)}.$$

Since for all $k'>k$ less than $\alpha_0$,  $\calW(k')$ is constructed over  $\calW(k)$,  we obtain a presentation of $ G_{W(\alpha_0)}$ by increasing union of the generating sets of  $G_{W(k)}$ (each of which contains that of   $G_{W(k_0)}$),     and by increasing union of the relators of  $G_{W(k)} $. The fifth point of Definition \ref{def;CW}  is then satisfied by  $\calW(\alpha_0)$, and it is a composite windmill constructed over $\calW(k_0)$. Since this is true for all $k_0<\alpha_0$, we obtain a contradiction with the definition of $\alpha_0$. 
\end{proof}

\subsubsection{Accessibility}

\begin{lemma} \label{lem;end_in_countable_time}

 Let $\calI$ be the set of countable ordinals $k$ such that $\forall k'<k, \calW(k')\subsetneqq  \calW(k)$. Then $\calI$  is countable. Moreover, for each $k_1, k_2$ in $\calI$, consecutive in  $\calI$,  there are at most $m$ ordinals between  $k_1$ and $k_2$.
\end{lemma}

\begin{proof}   For each $k\in \calI$, unless it is its maximal element, one can associate  its successor $s(k)$ in $\calI$, and therefore an element $X_{k}$ in $\bbY_*$ in $\calW(s(k))$ but not in  $\calW(k)$.  The assignation of $X_k$ is obviouly injective on $\calI$, and $\bbY_*$ is countable, thus $\calI$ is countable.

For the second assertion, assume that  there are  $m+1$ consecutive  countable ordinals $k_1, \dots, k_{m+1}$ outside $\calI$, all less than some $k_{t}\in \calI$. Then   by the pigeonhole argument, for two of them, $k, k'$, one has  $\calW(k) = \calW(k')$. Thus, by the rules of construction of $\calW(k+1)$, one has   $\calW(k) \subset  \calW(k+r) \subset \calW(k)$ for all $r\in \mathbb{N}$, ro equivalently,  for all $r$,   $\calW(k+r+1) \subset  \calW(k+r) \subset \calW(k+r+1)$.
  Since we take direct limits for limit ordinals, this holds also for all $r$ countable ordinal. However $k_{t}$ is a countable ordinal, and therefore     $\calW(k_{t}+1) \subset  \calW(k_{t}) \subset \calW(k_{t}+1)$,     contradicting that $k_{t} \in \calI$.

\end{proof}

\begin{lemma}\label{lem;the_top_k}
 There is a countable ordinal $k_{top}$, such that  $ \bbY_*\subset \calW(k_{top}) $.
\end{lemma}

\begin{proof}  By Lemma \ref{lem;end_in_countable_time}, the suppremum of $\calI$ is still a countable ordinal. Call  $k_{top}$ is this ordinal, $\calW(k_{top}) $ is thus well defined. Assume that  $ \bbY_* \not\subset \calW(k_{top}) $. Then it follows from Lemma \ref{lem;keepwalking}  
that  $\calW(k_{top}) $ is not $(\frac{c_*}{2}-20\kappa)$-convex. Therefore, there is a gap osculator in one of the coordinates, and this coordinate is reached while $r\leq m$. This is a contradiction on the definition of  $k_{top}$. Thus,   $ \bbY_* \subset \calW(k_{top}) $.

 \end{proof}

       \subsection{End of the proof  Theorems \ref{theo;main} and \ref{theo;mainGreendlingerLemma}}

       Consider $\calW(k_{top})$ from Lemma \ref{lem;the_top_k}. 
 Assume it is not a composite windmill. Then there is a smallest ordinal $k_1$ such that $\calW(k_{1})$ is not a composite windmill. If $k_1$ is not a limit ordinal, it is of the form $k_0 +1$ for $k_0$ such that $\calW(k_{0})$ is  a composite windmill. Lemma \ref{lem;still_a_CWM} concludes a contradiction. If $k_1$ is a limit ordinal, then  
  Lemma \ref{lem;still_a_CWM_again} concludes a contradiction. 
Thus  $\calW(k_{top})$ is a composite windmill.

  Since it contains all elements of $\bbY_*$, the statement of the Theorems \ref{theo;main} and  \ref{theo;mainGreendlingerLemma} follow  from the definition of composite windmill.

\section{Conclusion, application to Dehn twists, and Theorem \ref{theo;intro}}    

Let $\Sigma$ be an orientable closed surface of genus greater than
$2$. Consider $\MCG$ its Mapping Class Group. 

Bestvina Bromberg and Fujiwara produced a finite coloring of the set
of simple closed curves of $\Sigma$ such that two curves of same color
intersect, and a finite-index normal subgroup
$G_0$ of $\MCG$ that preserves the coloring. $G_0$ is called the color
preserving group. After refinement of the colors, we actually may
assume that the colors are in correspondance with the cosets of
$G_0$. We denote the colors by $\{1, \dots, m\}$.

Let $c$  and $c'$ be  simple closed curves. If they intersect, the
projection of $c'$ on $c$ is the family of elements in the arc complex
of the annulus around $c$ (that is the  cover of $\Sigma$ associated
to $c$) that come from lifts of $c'$. They are all disjoint. If $c''$
is another simple closed curve intersecting $c$, $d^\pi_c(c',c'')$ is
the diameter in the curve graph of the union of the projections of $c'$ and $c''$
on the annulus around $c$.

$d^\pi$ defines a composite projection system on the set of all
(homotopy classes of) simple
closed curves. Indeed, let $\Act(c)$ be the set of curves intersecting
$c$.  Clearly $d^\pi_c$ is symmetric,  and satisfies the
separation. The symetry in action, and the closeness in inaction are
also direct consequences of definitions. The finite filling property
is a consequence of the fact that   all sequences of subsurfaces up to isotopy,
increasing under inclusion, are eventually stationnary.   $d^\pi_c$ satisfies the
triangle inequality since it is a diameter of projections,  and the Behrstock inequality
\cite{Beh06}, see also \cite{Mang10} \cite{Mang13}. The properness is ensured by \cite[Lemma 5.3]{BBF}

We can now define two composite projection systems with composite
rotating families. The first one is defined on
$\bbY_*$ is the set $\mathfrak{S}$  of {\it all}  homotopy classes of simple closed curves of $\Sigma$.

Let us define $\bbY_i$ to be the subset of this set of simple closed
curve of color $i$ in the Bestvina-Bromberg-Fujiwara coloring, and
$\bbY_*$ their union. It is, as we just said, a composite projection
system on which $G_0$ acts by automorphisms. 

Performing the construction of \cite{BBF} and the choices as after Definition \ref{def;CPS}, we have
constants $\Theta, \kappa, c_*, \Theta_P, \Theta_{Rot}$.

We select $N_1$ such that all $N_1$-powers of Dehn
twists in $\MCG$ are in $G_0$. This is possible since there are only
finitely many $\MCG$-orbits of simple closed curves in $\Sigma$, and
$G_0$ has finite index. Then we select $N_2$ a multiple of $N_1$ such
that for all simple closed curve $c$, the Dehn twist $\tau_c^{N_2}$
around $c$ satisfies that $d_c (c', \tau^{N_2} c') >\Theta_{Rot}+2\Theta_P$ if
$c'$ is a curve of the same color than $c$ (hence intersecting $c$).
Since $d_c$ is comparable with $d^\pi_c$, by definition of the latter,
there exists such an exponent $N_2$.  Then it follows that, for all
$k\in \bbN$,  the
collection $\{\Gamma_c=\langle \tau_c^{kN_2}\rangle,\, c\in \mathfrak{S} \}$,
 is a composite
rotating family.

The second composite projection system is a sub-system, invariant for
$G_0$,  provided by the $\MCG$-orbit of a  
simple closed curve $c_0 \in \mathfrak{S}$. Namely, the composite rotating
family  is the collection
$\{\Gamma_c,\, c\in   (\MCG c_0) \subset \mathfrak{S} \}$.

It is straightforward that both families are composite rotating
families. 

One can then apply  Theorem  \ref{theo;main}. In the first case, one
obtains that the group generated by the $kN_2$-th powers of all Dehn
twists has a partially commutative presentation, which is the second
point of  Theorem \ref{theo;intro}. In the case of the second
composite rotating family,
one obtains that the group generated by all $kN_2$-th powers of all Dehn
twists that are $\MCG$-conjugated to $\tau_{c_0}$ has a partially commutative
presentation. This latter group is the normal closure of
$\tau_{c_0}^{kN_2}$ in $\MCG$.  We therefore obtained Theorem \ref{theo;intro}.

\section*{Acknowledgements}

The present work has been mostly developed during the visit of the
author to the Mathematical Science Research Institute in Berkeley,
during the thematic semester on Geometric Group Theory of the Fall
2016.  The author is supported by the Institut Universitaire de
France. 

I wish to thank M. Bestvina, K. Bromberg, 
K. Fujiwara,  J. Mangahas, J. Manning, and A. Sisto for discussions,  and  J. Tao,
and S. Dowdall for organising an influencial working seminar in
MSRI.    After I talked in MSRI on a first version of this work,
 which was performed on  cone-offs of the blown-up projection
 complexes of \cite{BBF}, and was more in
 line with \cite[\S 5]{DGO}, J. Mangahas
 suggested  that I work directly in the language of
projection complexes, which is indeed more natural for this situation,
and allows a similar argument; this choice is in line
with one of her work in progress with M. Clay and D. Margalit.  
I thank  T. Brendle and D. Margalit for suggesting relevant references. I finally thank the anonymous  referee  whose remarks helped to improve the paper.

{\small

\noindent
{\sc  Institut Fourier, Univ. Grenoble Alpes, cnrs, 38000 Grenoble, France}\\
{\tt e-mail: francois.dahmani@univ-grenoble-alpes.fr} 
}
\end{document}